\documentclass[11pt]{article}  
\usepackage[a4paper,margin=1in]{geometry}  
\usepackage{setspace}  
\usepackage{bm}

\usepackage{amsmath, amssymb, amsthm, graphicx, geometry}
\usepackage{tikz}
\usepackage{dsfont}
\usepackage{fix-cm}
\usepackage{enumitem}
\usepackage{bm} 
\usepackage{natbib}
\usepackage[hidelinks]{hyperref} 
\usepackage{multirow}
\usepackage{nicematrix} 
\usepackage{enumitem}
\setlist[enumerate]{after={\bigskip}}
\setlist[itemize]{after={\bigskip}}

\newlist{todolist}{itemize}{2}
\setlist[todolist]{label=$\square$}
\usepackage{pifont}
\newtheorem{definition}{Definition}
\newtheorem{proposition}{Proposition}
\newtheorem{remark}{Remark}
\newtheorem{lemma}{Lemma}
\newtheorem{theorem}{Theorem}
\newtheorem{corollary}{Corollary}
\newtheorem{assumption}{Assumption}

\usepackage{comment}
\newcommand{\ignore}[1]{}
\begin{document}
\begin{center}
    \vspace*{-2em}
    {\LARGE\bfseries Symmetry Testing in Time Series using Ordinal Patterns: A 
U-Statistic Approach\par}
    \vspace{1.5em}

    {\large
    \textbf{Annika Betken}$^{1}$,
    \textbf{Giorgio Micali}$^{1}$,
    \textbf{Manuel Ruiz Marín}$^{2}$\par}
    \vspace{1em}

    {\small
    $^{1}$University of Twente, Department of Applied Mathematics, \\
Enschede, Netherlands\\[0.5em]
    $^{2}$Departamento de Métodos Cuantitativos, Ciencias Jurídicas y Lenguas Modernas,\\
    Universidad Politécnica de Cartagena, Spain\\[0.8em]
    \texttt{\{a.betken, g.micali\}@utwente.nl, manuel.ruiz@upct.es}\par}
    \vspace{2em}
\end{center}

\begin{abstract}
     We introduce a general framework for testing temporal symmetries in time series based on the distribution of ordinal patterns. While previous approaches have focused on specific forms of asymmetry, such as time reversal, our method provides a unified framework applicable to arbitrary symmetry tests. We establish asymptotic results for the resulting test statistics under a broad class of stationary processes. Comprehensive experiments on both synthetic and real data demonstrate that the proposed test achieves high sensitivity to structural asymmetries while remaining fully data-driven and computationally efficient.
\end{abstract}

\section{Introduction}

Symmetry plays a central role in time series analysis. Many important classes of processes, most prominently stationary Gaussian processes, exhibit invariances such as time reversibility and reflection symmetry. These invariances ground  model assumptions and determine the validity of inferential procedures. Their violation, in contrast, indicates that the assumed model hypotheses are not respected, pointing to other structural effects. 
Testing for symmetry is therefore a natural and essential diagnostic.  A particularly important case is time reversibility, the invariance of a process under temporal inversion. Formally, a time series $(X_t)_{t\in \mathbb{Z}}$ is time reversible if
\[
(X_{t_1},\ldots,X_{t_m}) \stackrel{\mathcal{D}}{=} (X_{t_m},\ldots,X_{t_1})
\quad \text{for all } m\in \mathbb{N},t_1,\ldots,t_m\in \mathbb{Z},
\]
that is, its finite-dimensional distributions are invariant under time reversal; see \cite{Weiss_1975}.
 For example, electroencephalographic (EEG) recordings obtained during normal brain activity are typically close to time-reversible, whereas EEG recordings during epileptic seizures (brief episodes of abnormal, highly synchronized neuronal firing that disrupt normal brain function) exhibit pronounced temporal irreversibility; see \cite{VANDERHEYDEN1996283}. Similarly, the heart rate variability of healthy subjects provides another example of an irreversible process, as we will further discuss in Section~\ref{sec:simulations}.

In recent years, testing methods based on \emph{ordinal patterns} have become a cornerstone of statistical approaches to time irreversibility. Building on the ordinal patterns framework, several works have proposed metrics for detecting time irreversibility; see~\cite{martinez2018time}, \cite{zanin2018assessing,yao2019quantifying,yao2021time}. The aim of this article is to develop  a framework that generalizes these approaches: we provide a unified, mathematically grounded formulation of symmetry testing for time series, of which classical time-reversibility tests arise as a special case. 

Ordinal patterns, introduced by~\cite{Bandt-Pompe}, provide a symbolic representation of time series that captures the relative ordering of consecutive observations rather than their magnitudes. 
Let $\mathcal S_d$ denote the symmetric group on $\{1,\ldots,d\}$, i.e., the set of all permutations of $\{1,\ldots,d\}$. For a block of $d$ consecutive data points from a process $(X_t)_{t\in\mathbb{Z}}$, $
 (X_t, \ldots, X_{t+d-1}) \in \mathbb{R}^d, $
the ordinal pattern is defined as the permutation $\pi\in\mathcal S_d$ obtained
by ordering the components of the block. More precisely, the ordinal patterns
map is
\begin{equation}\label{eq:def_op}
\Pi:\mathbb{R}^d \to \mathcal{S}_d,
\qquad (X_t,\ldots,X_{t+d-1}) \mapsto \pi=(\pi_1,\ldots,\pi_d),
\end{equation}
where $\pi$ is the unique permutation of $(1, 2, \ldots, d)$ such that
\[
X_{t+\pi_1-1} \le \cdots \le X_{t+\pi_d-1},
\] i.e.,
$\pi_j$
represents the rank of $X_{t+j-1}$  within the block $(X_t,\ldots,X_{t+d-1})$. 
In case of ties, i.e., when $X_i=X_j$ for some $i\neq j$, we break them by ordering indices inversely, e.g.\ $\Pi([1,2,2])=(1,3,2)$. 
Since throughout this work we consider processes with continuous distributions, ties occur with probability zero (for other ways of handling ties we refer to \cite{SCHNURRties}). 
For instance, let us consider the time series
\(\{4.2,\; 3.1,\; 5.0,\; 6.3,\; 2.9,\; 7.1,\; 1.8,\; 3.7,\ldots\}
\)
and  $d=3$, so that each ordinal pattern is extracted from a window of three consecutive values. 
The first window $(4.2,\,3.1,\,5.0)$ is sorted as $(3.1,\,4.2,\,5.0)$, yielding the permutation $\pi=(2,1,3)$. 
The next window $(3.1,\,5.0,\,6.3)$ gives $\pi=(1,2,3)$, and so on. 
Thus, the time series is transformed into a sequence of ordinal patterns
\(
\{(2,1,3),\; (1,2,3),\; (3,1,2),\; (2,1,3),\ldots\},
\)
each representing the relative ordering of local observations. This symbolic representation offers several advantages: it is invariant under strictly increasing transformations of the data (such as rescaling or nonlinear monotone mappings), robust to noise since it depends only on relative orders rather than magnitudes, and computationally efficient. Throughout this work, we consider stationary processes $(X_t)_{t\in \mathbb{Z}}$, and, under this assumption, the symbolic transformation induces a time-independent probability distribution on $\mathcal{S}_d$ that captures the dependence structure of the underlying process:
\begin{equation}\label{eq:pattern_prob}
p(\pi) := \mathbb{P}\!\left( \Pi(X_t,\ldots,X_{t+d-1}) = \pi \right),
\qquad \pi \in \mathcal{S}_d.
\end{equation}
Note, however, that this property continues to hold under weaker assumptions. In fact, stationarity of ordinal patterns arises already when the process has stationary increments, i.e., when $\Delta_t = X_t - X_{t-1}$ forms a stationary sequence, which is satisfied by a broader class of models since stationarity of $(X_t)_{t\in \mathbb{Z}}$ implies stationarity of its increments. 

Crucially, ordinal patterns are particularly suited for symmetry testing. In fact, in addition to time reversibility, we consider \emph{reflection symmetry}, defined as follows: $(X_t)_{t\in\mathbb Z}$ is said to be
reflection symmetric if
\[
(X_{t_1},\ldots,X_{t_m})
\;\stackrel{\mathcal D}{=}\;
(-X_{t_1},\ldots,-X_{t_m})
\qquad
\text{for all } m\in\mathbb N \text{ and } t_1,\ldots,t_m\in\mathbb Z.
\]
  Lemma~1 in \cite{sinn2011estimation} shows that if a time series is stationary Gaussian, then its ordinal pattern distribution $(p(\pi))_{\pi \in \mathcal{S}_d}$ is invariant under time reversal and reflection. 
 In other words, symmetries of a time series are mirrored in corresponding invariances of its ordinal patterns distribution. 
Based on this intuition, we formalize symmetry testing by specifying the following testing problem for a partition $\mathcal{G}$ of $\mathcal{S}_d$:
\begin{align*}
 \mathcal{H}_0:& \; p(\pi_i) = p(\pi_j) 
    \qquad \forall\, \pi_i,\pi_j \in G, \; G \in \mathcal{G}, 
\\
\mathcal{H}_1: & \; \exists\, G \in \mathcal{G}, \; \pi_i,\pi_j \in G 
    \text{ with } p(\pi_i) \neq p(\pi_j).
\end{align*}
When $\mathcal{G}$ encodes the symmetries of Gaussian processes, the rejection of $\mathcal H_0$ excludes 
Gaussian models as plausible data generating
mechanisms. 
More generally, different choices of $\mathcal{G}$ allow one to target time reversibility, directional asymmetry, or other structural features.  To formalize the relevant symmetries, we introduce two natural operations on ordinal patterns; see \cite{sinn2011estimation}.   
For $\pi=(\pi_1,\ldots,\pi_{d}) \in \mathcal{S}_d$, define 
\begin{align*}
\pi^{\mathrm{rev}} &:= (\pi_{d}, \pi_{d-1}, \ldots, \pi_{1}), 
\qquad \text{(time reversal)}, \\
\pi^{\mathrm{ref}} &:= (d-\pi_1,\, d-\pi_2,\, \ldots,\, d-\pi_{d}),
\qquad \text{(reflection)}.
\end{align*}
That is, $\pi^{\mathrm{rev}}$ corresponds to reversing the temporal order of the block,
while $\pi^{\mathrm{ref}}$ corresponds to flipping the relative ranks within the block.
Different choices of the partition $\mathcal{G}$ then lead to tests of different structural
hypotheses, as summarized in Table \ref{tab:partitions}.
\begin{table}[h!]
\centering
\begin{tabular}{lll}
\textbf{Partition $\mathcal{G}$} & \textbf{Null hypothesis $\mathcal{H}_0$} & \textbf{Targeted property} \\ \hline \\[-0.8em]
$\{\{\pi,\pi^{\mathrm{rev}}\} : \pi \in \mathcal{S}_d\}$ 
    & $p(\pi)=p(\pi^{\mathrm{rev}})$ 
    & Time reversibility \\
$\{\{\pi,\pi^{\mathrm{ref}}\} : \pi \in \mathcal{S}_d\}$ 
    & $p(\pi)=p(\pi^{\mathrm{ref}})$ 
    & Reflection symmetry \\
$\{\{\pi,\pi^{\mathrm{rev}}, \pi^{\mathrm{ref}}\} : \pi \in \mathcal{S}_d\}$ 
    & Equal prob. within each group 
    & Gaussian symmetry
\end{tabular}
\caption{Examples of partitions $\mathcal{G}$ of $\mathcal{S}_d$ and the corresponding
structural hypotheses tested.}
\label{tab:partitions}
\end{table}
We emphasize that the proposed test assesses equality of ordinal patterns probabilities
within the groups specified by the partition $\mathcal{G}$.
A rejection of  $\mathcal{H}_0$ therefore implies a violation of the corresponding
structural symmetry encoded by $\mathcal{G}$.
Conversely, not rejecting the null should not be interpreted as a
characterization of a specific model class, since the same symmetry relations may
also hold for non-Gaussian or non-linear processes. In particular, the partition based on joint reversal and reflection symmetries
captures symmetries that are satisfied by stationary Gaussian processes, but it is not exclusive to this class.

This article is organized as follows. 
In Section~\ref{sec:main} we present the test statistics for the testing problem ($\mathcal{H}_0$,  $\mathcal{H}_1$) introduced above, and we establish our main theoretical results: Theorem~\ref{theorem:asymptotic_distribution} provides the asymptotic distribution under  $\mathcal{H}_0$, while Theorem~\ref{thm:alt-limit} characterizes the limiting behavior under $\mathcal{H}_1$. 
Section~\ref{sec:simulations} reports a comprehensive simulation study illustrating the finite-sample performance of the proposed tests, followed by applications to real data examples. 
Finally, Section~\ref{sec:conclusion} concludes with a summary and possible directions for future research.

\subsection*{Notation}

Vectors and matrices are denoted in boldface. We denote by \( (X_t)_{t \in \mathbb{Z}} \) a generic zero mean stationary time series, and define the associated vectorized series of length \( d\geq 1 \) as \( \bm{X}_t = (X_t, \ldots, X_{t+d-1})^\top \). The cumulative distribution function of a random vector $(X_1, \ldots, X_d)^\top$ is denoted as $F(\bm{x})
:= \mathbb{P}\big( X_{1}\le x_1,\ldots,X_{d}\le x_d \big),
\, \bm{x}=(x_1, \ldots,x_d)^\top\in\mathbb{R}^d .$ 
The indicator of an event $A$ is denoted by  $\mathds{1}(A)\;.$ The symbol $\xrightarrow{\mathcal{D}}$ 
 denotes convergence in distribution. For two real sequences $(a_n)_{n\ge 1}$ and $(b_n)_{n\ge 1}$ with $b_n>0$, 
we write
$a_n = \mathcal{O}(b_n)
\quad \text{as } n \to \infty $
if there exist constants $C>0$ and $n_0 \in \mathbb{N}$ such that $
|a_n| \le C\, b_n
\quad \text{for all } n \ge n_0.$ Similarly, for a sequence of random variables $(X_n)_{n\ge 1}$ and a sequence
$(b_n)_{n\ge 1}$ with $b_n>0$, we write
$X_n = \mathcal{O}_{\mathbb P}(b_n)
\quad \text{as } n \to \infty$
if for every $\varepsilon>0$ there exist constants $C_\varepsilon>0$ and
$n_\varepsilon\in\mathbb N$ such that $
\mathbb P\!\left( |X_n| > C_\varepsilon\, b_n \right) < \varepsilon
\quad \text{for all } n \ge n_\varepsilon $.

\section{Main Results}
\label{sec:main}
In this section we define a suitable test statistic for  deciding on the test problem  $\mathcal{H}_0$ versus $\mathcal{H}_1$ and establish its asymptotic distribution under  $\mathcal{H}_0$ (Theorem \ref{theorem:asymptotic_distribution}) and  $\mathcal{H}_1$ (Theorem \ref{thm:alt-limit}). 
\subsection{Test Statistic}

In order to decide on the test problem  $\mathcal{H}_0$ versus $\mathcal{H}_1$, we quantify 
departures from symmetry in the ordinal patterns distribution.  Let $\mathcal{G}=\{G_1,\ldots,G_m\}$ be a partition of $\mathcal{S}_d$. 
We define the
symmetrized distribution with respect to $\mathcal{G}$ as
\[
p_\mathcal{G}(\pi) = \frac{1}{|G|} \sum_{\pi_j \in G} p(\pi_j),
\qquad \pi \in G \in \mathcal{G},
\]
which, by construction, satisfies $p_\mathcal{G}(\pi_i)=p_\mathcal{G}(\pi_j)$ 
for all $\pi_i,\pi_j \in G$.  Under the null hypothesis $\mathcal H_0$, this distribution coincides
with the ordinal patterns distribution $(p(\pi))_{\pi\in\mathcal S_d}$.
To measure the discrepancy between $(p(\pi))_{\pi\in \mathcal{S}_d}$ and $(p_\mathcal{G}(\pi))_{\pi\in \mathcal{S}_d}$, 
we consider the difference
\begin{equation}
D_2(\mathcal{G}) 
:= \sum_{i=1}^{d!} \big(p_\mathcal{G}(\pi_i)\big)^2 
   - \sum_{i=1}^{d!} \big(p(\pi_i)\big)^2.
\label{eq:D2}
\end{equation}
Then, it follows by Jensen's inequality that 
\begin{align*}
\sum_{i=1}^{d!} \big(p_\mathcal{G}(\pi_i)\big)^2
&= \sum_{k=1}^m |G_k|\,
   \Big( \frac{1}{|G_k|} \sum_{\pi \in G_k} p(\pi) \Big)^2\leq \sum_{k=1}^m \sum_{\pi \in G_k} \big(p(\pi)\big)^2
= \sum_{i=1}^{d!} \big(p(\pi_i)\big)^2.
\end{align*}
Therefore, $D_2(\mathcal{G}) \leq 0$ with equality if
$p$ is symmetric with respect to $\mathcal{G}$.  This motivates using $-D_2(\mathcal{G})$ as a population measure of asymmetry:  larger values indicate stronger deviations from symmetry, while 
$D_2(\mathcal{G}) = 0$ under the null hypothesis.    
In practice, since $(p(\pi))_{\pi\in \mathcal{S}_d}$ is unknown, $D_2(\mathcal{G})$ must be estimated. Relation~(4) in \cite{Schnurr15092025} shows that, for $F$ denoting the cumulative distribution function of $\bm{X}_1$:
\begin{equation}\label{eq:SCI-definition}
    \sum_{i=1}^{d!} \big(p(\pi_i)\big)^2
    \;=\;
    \int_{\mathbb{R}^d} \int_{\mathbb{R}^d}
    \mathds{1}\!\left\{\Pi(\bm{x})=\Pi(\bm{y})\right\}
    \, dF(\bm{x})\, dF(\bm{y})\;.
\end{equation}
 \cite{Schnurr15092025} propose the following $U$-statistic as an unbiased estimator of the right-hand side of \eqref{eq:SCI-definition}:
\begin{equation}\label{eq:SCI-estimator}
    S_n^d
    =\frac{2}{n(n-1)}
      \sum_{1\le j<k\le n}
      \mathds{1}\!\left\{\Pi(\bm{X}_j)=\Pi(\bm{X}_k)\right\}\;.
\end{equation}
The integral in \eqref{eq:SCI-definition} is known as \emph{symbolic correlation integral}, introduced in
\cite{caballero2019symbolic}.  To estimate the first term of \eqref{eq:D2}, for $\pi_i \in G \in \mathcal{G}$, we consider the empirical  and the symmetrized empirical frequencies
\begin{equation}
\label{eq:group_estimator}
\hat p_n(\pi) = \frac{1}{n} \sum_{t=1}^n 
\mathds{1} \left( \Pi(X_t,\ldots,X_{t+d-1}) = \pi \right)\;,
\qquad 
\hat p_\mathcal{G}(\pi_i) 
= \frac{1}{|G|} \sum_{\pi \in G} \hat p_n(\pi)\;.
\end{equation}
Replacing each term in \(D_2(\mathcal{G})\) by its empirical
version yields the statistic
\begin{equation}
\hat D_2(\mathcal{G}) 
= \sum_{i=1}^{d!} \big(\hat p_\mathcal{G}(\pi_i)\big)^2 - S_n^d.
\label{eq:estimator_asymmetry}
\end{equation}
Proposition~\ref{prop:decomposition} in Appendix~\ref{appendix:auxiliary_results}
shows that the asymptotic behavior of \eqref{eq:estimator_asymmetry} can be
characterized via an appropriate $U$-statistic representation, up to a remainder term.
This observation forms the basis for the derivation of the limiting
distributions in the next section.
\subsection{Limiting Distributions}
We now derive the asymptotic behavior of the test statistic
$\hat D_2(\mathcal G)$ under the null and alternative hypotheses.
A key step consists in expressing $\hat D_2(\mathcal G)$ as a
$U$-statistic plus a remainder term of smaller order.
To fix notation, we briefly recall the notion of a $U$-statistic.
A function $h:\mathbb{R}^m \to \mathbb{R}$ is said to be \emph{symmetric} if
\[
h(x_1,\ldots,x_m)
=
h(x_{\pi(1)},\ldots,x_{\pi(m)})
\]
for all $(x_1,\ldots,x_m)\in\mathbb{R}^m$ and all permutations $\pi$ of
$\{1,\ldots,m\}$.
Such a function is referred to as a \emph{kernel} of order $m$.
\begin{definition}[U-statistics]
\label{def:ustat}
Let $h:\mathbb{R}^m \to \mathbb{R}$ be a symmetric kernel of order $m$.
The associated $U$-statistic based on a sample $X_1,\ldots,X_n$ is
\[
U_{n,m}(h)
=
\binom{n}{m}^{-1}
\sum_{1 \le i_1 < \cdots < i_m \le n}
h(X_{i_1},\ldots,X_{i_m}).
\]
It is called \emph{degenerate} if
\[
h_1(x)
:=
\mathbb{E}[h(x,X_2,\ldots,X_m)]
=
0
\quad \text{for all } x \in \mathbb{R}.
\]
\end{definition}
For general background on $U$-statistics and their asymptotic theory,
we refer to \cite{Denker1985,Lee1990,Dehling2006}.
We now return to our test statistic.
Proposition~\ref{prop:decomposition} in
Appendix~\ref{appendix:auxiliary_results} shows that
\[
\hat{D}_2(\mathcal{G})
=
\frac{1}{n^2}
\sum_{1 \le k_1 < k_2 \le n}
h(\bm X_{k_1}, \bm X_{k_2})
+
\mathcal{O}_{\mathbb P}(n^{-1})
=:
U_n^{d}
+
\mathcal{O}_{\mathbb P}(n^{-1}),
\]
where $U_n^{d}$ denotes the $U$-statistic of order two associated with the
symmetric kernel
\[
h(\bm x, \bm y)
=
\sum_{G \in \mathcal G}
\frac{1}{|G|}
\,
\mathds{1}\!\left\{
\Pi(\bm x)\in G,\,
\Pi(\bm y)\in G
\right\}
-
\mathds{1}\!\left\{
\Pi(\bm x)=\Pi(\bm y)
\right\}.
\]

Under the null hypothesis $\mathcal H_0$, the kernel $h$
is degenerate (see Proposition~\ref{prop:kernel-properties} in
Appendix~\ref{appendix:auxiliary_results}).
Consequently, the asymptotic distribution of $\hat D_2(\mathcal G)$
under $\mathcal H_0$ follows from the limit theory for degenerate
$U$-statistics developed in \cite{carlstein1988degenerate}.
In particular, the limit law is a generalized chi-squared distribution,
with weights given by the eigenvalues of the Hilbert--Schmidt operator
induced by the kernel $h$.
In other words, the distribution of $\hat{D}_2(\mathcal{G})$ under the  $\mathcal{H}_0$ depends on the eigenvalues and eigenfunctions of the integral operator:
\begin{equation}
    \mathcal{A}: L^2(\mathbb{R}^d, dF) \longrightarrow L^2(\mathbb{R}^d, dF),
\qquad
(\mathcal{A}[g])(\bm{u}) := \int_{\mathbb{R}^d} h(\bm{u},\bm{v})\, g(\bm{v})\, dF(\bm{v}).
\label{eq:operator} \tag{$\mathcal{A}$}
\end{equation}
The integral operator \ref{eq:operator}
is well defined, since $h(\bm{u},\bm{v})$ is bounded and piecewise constant in $\bm{v}$,
which implies that $\mathcal{A}[g]\in L^2(\mathbb{R}^d,dF)$ for all
$g\in L^2(\mathbb{R}^d,dF)$. 
A full characterization of the eigenvpairs  $(\lambda, g)$ is given in Proposition \ref{proposition:eigenvalues} in Appendix \ref{appendix:auxiliary_results}. In particular, the operator admits $d!-m$ nonzero eigenvalues,
counted with multiplicity, and $d!-m$ corresponding eigenfunctions,
denoted by $g^{(1)}, \ldots, g^{(d!-m)}\;.$

 We introduce the following notation. For $\mathcal{G}=\{G_1, \ldots, G_m\}$, define $|G_1|=d_1, \ldots, |G_m|=d_m$, i.e., $\sum_i d_i=d!$, and $G_i=\{\pi_{i,1}, \ldots, \pi_{i,{d_i}}\}$. For $\pi_{i,1},\pi_{i,2} \in G_i$, $ p_\mathcal{G}(\pi_{i,1})=p_\mathcal{G}(\pi_{i,2})$, thus we define $p_i$ as $p_i:=p_\mathcal{G}(\pi) $ for any $\pi \in G_i\;.$ Under the null hypothesis $p_i=p(\pi_i)\;.$ Lastly, let $F$ denote the common marginal distribution of the $d$-dimensional blocks $\bm{X}_t$, and let $F_i$ be the joint distribution of $(\bm{X}_0,\bm{X}_i)$. 
The notation $F_i \ll F \times F$ means that the joint law of $(\bm{X}_0,\bm{X}_i)$ is absolutely continuous with respect to the product measure of the marginals. We further assume that the process $(\bm X_t)_{t\in\mathbb Z}$ is
\emph{$\alpha$-mixing} (strong mixing), 
meaning that the dependence between
events separated by a large time lag vanishes asymptotically.
More precisely:
\begin{definition}
\label{def:alpha-mixing}
Let $(\bm X_t)_{t\in\mathbb Z}$ be a process defined on a probability space
$(\Omega, \mathcal{F}, \mathbb{P})$. For $k \leq l$, define the $\sigma$-fields
$\mathcal{A}_k^l := \sigma(\bm{X}_k,\ldots,\bm{X}_l)$.
The process $(\bm{X}_t)_{t \in \mathbb{Z}}$ is called \emph{strongly mixing}
(or \emph{$\alpha$--mixing}) if $\alpha_n \to 0$ as $n \to \infty$, where
\[
\alpha_n
=
\sup_{m \in \mathbb{Z}}
\left\{
\sup
\Big|
\mathbb{P}(A \cap B) - \mathbb{P}(A)\mathbb{P}(B)
\Big|
\right\},
\]
and the last supremum is taken over all
$A \in \mathcal{A}_{-\infty}^m$
and
$B \in \mathcal{A}_{m+n}^\infty$.
\end{definition}
The coefficient $\alpha_n$ quantifies how strongly the past of the process $(\bm{X}_t)_{t\le m}$ still influences its future, $(\bm{X}_t)_{t\ge m+n}$, when they are separated by $n$ time steps. The condition $\alpha_n\to0$ states that these influences become asymptotically independent as $n\to\infty$. This is a standard weak-dependence assumption, which is needed for developing central limit theorems in time-series settings, see for instance \cite{ibragimov1971independent}. 
A strictly stationary process that is $\alpha$-mixing is ergodic, see section 2.5 in \cite{Bradley2005}.
\begin{theorem}[under  $\mathcal{H}_0$]
\label{theorem:asymptotic_distribution}
Let $t = d! - m$, and let $(\lambda_1, g^{(1)}), \ldots, (\lambda_t, g^{(t)})$ denote the eigenvalue–-eigenfunction pairs  of the
Hilbert--Schmidt operator \ref{eq:operator} associated with the kernel $h$. 
Assume that $F_i \ll F \times F$ for all $i\ge 1$, and that there exists
$\delta>0$ such that $
\mathbb{E}\bigl[\,|g^{(i)}(\bm{X}_0)|^{2+\delta}\,\bigr] < \infty $, $ i=1,\ldots,t. $ Further, let $(\bm{X}_t)_{t\geq 1}$ be $\alpha$-mixing with mixing coefficients  $(\alpha_k)_{k\ge 1}$ satisfying the 
summability condition $
\sum_{k=1}^\infty \alpha_k^{\delta/(2+\delta)} < \infty  $. Then,
\begin{equation}
\label{eq:quadratic1}
n \,\hat{D}_2(\mathcal{G})
\;\xrightarrow{\mathcal{D}}\;
\sum_{i=1}^t \lambda_i \,(W_i^2 - 1)+c
\; ,
\end{equation}
where $W_1, \ldots, W_t $ are centered Gaussian random variables with covariances
\begin{equation}
    \label{eq:cov_entry}
    \mathbb{E}[W_i W_j]
\;=\; \sum_{k=-\infty}^{\infty}
\mathbb{E}\!\left[\, g^{(i)}(\bm{X}_0)\, g^{(j)}(\bm{X}_k) \right], \qquad1\leq i,j\leq t\;,
\end{equation}
and where
\[
c
\,=\,
\sum_{G\in\mathcal{G}}
\frac{1}{|G|}
\,\mathbb{P}\!\left(\Pi(\bm{X}_1)\in G\right)
\;-\;
\int_{\mathbb{R}^d}\!\!\int_{\mathbb{R}^d}
\mathds{1}\!\left\{\Pi(\bm{x})=\Pi(\bm{y})\right\}
\,dF(\bm{x})\, dF(\bm{y}) .
\]
\end{theorem}
 The proof can be found in Appendix \ref{appendix:proofs}. We use the notation $S_t(\bm{p}):=\sum_{i=1}^t \lambda_i \,(W_i^2 - 1) $ following the original convention in \cite{carlstein1988degenerate}. Throughout the paper the subscript $t$ in $S_t(\bm{p}) $  refers to the quantity $t = d! - m $. Note that this use of $t$ is independent of the time index used for the underlying time series. 
\begin{remark}
\begin{enumerate}
\item []
    \item Proposition~\ref{prop:kernel-properties} in Appendix \ref{appendix:auxiliary_results} shows that $\mathbb{E}[h(\bm{Y}_1,\bm{Y}_2)]=0$ under  $\mathcal{H}_0$, so the centering term does not appear in Theorem~\ref{theorem:asymptotic_distribution}. 
    Moreover, by Lemma~\ref{lemma:mean_zero_eigenfunction} in Appendix \ref{appendix:auxiliary_results}, 
$\mathrm{Cov}(g^{(i)}(\bm{X}_s),g^{(j)}(\bm{X}_\ell)) = \mathbb{E}[g^{(i)}(\bm{X}_s)g^{(j)}(\bm{X}_\ell)]$.  
This yields the alternative representation
\begin{equation*}
\mathbb{E}[W_i W_j] = \lim_{n\to \infty} \frac{1}{n}\sum_{s,\ell=1}^n \mathbb{E}[g^{(i)}(\bm{X}_s) g^{(j)}(\bm{X}_\ell)] 
= \sum_{\ell=-\infty}^\infty \mathbb{E}[g^{(i)}(\bm{X}_0) g^{(j)}(\bm{X}_\ell)],
\end{equation*}
 In fact, using the stationarity of the process $(X_t)_{t\in \mathbb{Z}}$
    \begin{align*}
        \frac{1}{n} \sum_{s,\ell=1}^n \mathbb{E}[g^{(i)}(\bm{X}_s) g^{(j)}(\bm{X}_\ell)] =\sum_{\ell=-n+1}^{n-1} \left( 1-\frac{|\ell|}{n} \right) \mathbb{E} [g^{(i)}(\bm{X}_0) g^{(j)}(\bm{X}_\ell)]\;.
    \end{align*}
    by taking the limit for $n \to \infty $ on both sides, we obtain $ \mathbb{E}[W_i W_j]$  
which matches the covariance structure in Theorem~\ref{theorem:asymptotic_distribution}.
\item The condition $F_i \ll F \times F$ holds under very mild assumptions; for instance, it is satisfied when the underlying innovations $(Z_t)_{t\in \mathbb{Z}}$ form an i.i.d.\ sequence with a continuous distribution, which is the typical modeling setting for standard processes such as ARMA models.
\end{enumerate}
\end{remark}
Under $\mathcal{H}_1$, the kernel $h$ is non-degenerate (see Proposition \ref{prop:kernel-properties} in Appendix \ref{appendix:auxiliary_results}) and $\hat{D}_2(\mathcal{G})$ behaves like a standard $U$-statistic with a non-trivial asymptotic distribution. Thus, the limiting distribution of $\hat{D}_2(\mathcal{G})$ under $\mathcal{H}_1$ follows directly from Theorem~B.7 in \cite{Schnurr15092025}, applied to the non-degenerate $U$-statistic $U_n^d$, once the 1-continuity of the kernel $h$ has been established.
The latter property is verified in Appendix \ref{appendix:auxiliary_results}.  The notion of 1-continuity, see \cite{borovkova2001limit}, states that for a measurable function, sufficiently small perturbations of its argument
lead to correspondingly small variations of its value in an
$L^{p}(\Omega,\mathcal{F},\mathbb{P})$ sense with respect to the underlying
distribution.
In other words, the function is controlled in mean when its arguments are close.  See also the detailed discussion in \cite{Schnurr15092025}, Page 9.
\begin{definition}
\label{def:p-continuity}
Let $F$ be a probability distribution on $\mathbb{R}^d$, and $p\geq 1$.
A measurable function $g : \mathbb{R}^d \to \mathbb{R}^m$ is called 
\emph{$p$-continuous with respect to $F$} if there exists $\varphi:(0,\infty)\to(0,\infty)$ 
with $\varphi(\varepsilon)=o(1)$ as $\varepsilon \to 0$ such that
\[
\mathbb{E}\left[ \| g(\mathbf{Y}) - g(\mathbf{Y}') \|^p \; \mathds{1}(\{\|\mathbf{Y}-\mathbf{Y}'\|\leq \varepsilon\}) \right] 
\leq \varphi(\varepsilon),
\]
for all random vectors $\mathbf{Y},\mathbf{Y}'$ with distribution $F$.
\end{definition}
 Unlike in the previous theorem, where $\alpha$-mixing assumptions were sufficient,
the analysis under $\mathcal H_1$ requires the notion of weak dependence,
namely \emph{$\beta$-mixing} (absolute regularity). 
This  condition is needed in order to apply the asymptotic theory for
non-degenerate $U$-statistics developed in \cite{Schnurr15092025}.
\begin{definition}
\label{def:beta-mixing}
Let $(Z_t)_{t \in \mathbb{Z}}$ be a process defined on a probability space $(\Omega, \mathcal{F}, \mathbb{P})$. For $k \leq l$, define the $\sigma$-fields 
$\mathcal{A}_k^l := \sigma(Z_k,\ldots,Z_l)$. 
The process $(Z_t)_{t \in \mathbb{Z}}$ is called \emph{absolutely regular} if
$\beta_n \to 0$ as $n \to \infty$, where
\[
\beta_n = \sup_{m \in \mathbb{Z}}  \left\{
\sup \sum_{i=1}^I \sum_{j=1}^J 
\Big| \mathbb{P}(A_i \cap B_j) - \mathbb{P}(A_i)\mathbb{P}(B_j) \Big| \right\},
\]
and the last supremum is taken over all finite 
$\mathcal{A}_{-\infty}^m$-measurable partitions $(A_1,\ldots,A_I)$ and 
$\mathcal{A}_{m+n}^\infty$-measurable partitions $(B_1,\ldots,B_J)$.
\end{definition}

Furthermore, to establish the asymptotic distribution of the test statistics under the alternative we need to assume that the considered time series is an $r$-approximating functional of an absolutely regular time series. This framework reflects many common modeling situations in which the observed process $(X_t)_{t\in\mathbb{Z}}$ is generated from an underlying sequence $(Z_t)_{t\in\mathbb{Z}}$ through a measurable transformation. The following definition makes this precise.
\begin{definition} 
\label{def:functional}
Let $(Z_t)_{t \in \mathbb{Z}}$ be a stationary time series taking values in a measurable space $S$.  
A $\mathbb{R}^d$-valued sequence $(\bm{X}_t)_{t \in \mathbb{Z}}$ is called a \emph{functional} of $(Z_t)_{t \in \mathbb{Z}}$ 
if there exists a measurable function $f : S^{\mathbb{Z}} \to \mathbb{R}^d$ such that
\[
\bm{X}_t = f((Z_{t+k})_{k \in \mathbb{Z}}) \quad \text{for all } t \in \mathbb{Z}.
\]
\end{definition}
For example, if $X_t=\sum_{j=0}^\infty a_j Z_{t-j}$, and the series is well defined (e.g., $\sum_j |b_j|<\infty$, and finite second moment of $X_t$), then $(X_t)_{t\in\mathbb{Z}}$ is a functional of $(Z_t)_{t\in\mathbb{Z}}$.
In fact, by definying the measurable map
\begin{equation*}
    f:\mathbb{R}^{\mathbb{Z}} \to \mathbb{R}, \qquad
f\big( (z_k)_{k\in\mathbb{Z}} \big)
:= \sum_{j=0}^\infty b_j z_{-j}\;,
\end{equation*}
then
$X_t = f\big((Z_{t+k})_{k\in\mathbb{Z}}\big)$ for $t\in\mathbb{Z},$
thus $(X_t)_{t\in \mathbb{Z}}$ is precisely the functional obtained by applying $f$ to each shifted trajectory of $(Z_t)_{t\in \mathbb{Z}}$. 

While Definition \ref{def:functional} characterizes how $(X_t)_{t\in \mathbb{Z}}$ is generated from $(Z_t)_{t\in \mathbb{Z}}$, for asymptotic theory we additionally need a way to quantify how strongly $(X_t)_{t\in \mathbb{Z}}$ depends on values of $(Z_t)_{t\in \mathbb{Z}}$ that are far in the past or the future. The following definition of an $r$-approximating functional captures exactly this. It requires that conditioning on a finite window $(Z_{t-k},\dots,Z_{t+k})$ yields an accurate approximation of $X_t$. In other words, the influence of remote innovations must decay sufficiently fast. 
\begin{definition}
\label{def:r-approx}
Let $(X_t)_{t \in \mathbb{Z}}$ be a functional of $(Z_t)_{t \in \mathbb{Z}}$, and let $r \geq 1$.  
Suppose $(a_k) _{k\geq 0}$ are constants with $a_k \to 0$.  
We say $(X_t)_{t \in \mathbb{Z}}$ is an \emph{$r$-approximating functional} if
\[
\mathbb{E}\Big\| \bm{X}_0 - \mathbb{E}(\bm{X}_0 \mid \mathcal{A}_{-k}^k) \Big\|^r \leq a_k
\quad \text{for all } k \in \mathbb{N}_0.
\]
The sequence $(a_k) _{k\geq 0}$ is said to be of \emph{size} $-\lambda$ if $a_k = \mathcal{O}(k^{-\lambda-\epsilon})$ for some $\epsilon > 0$.
\end{definition}
 Returning to the linear process example, Definition \ref{def:r-approx} states precisely that truncating the infinite series $ X_t = \sum_{j=0}^\infty b_j Z_{t-j}$
after $k$ terms provides an accurate approximation of $X_t$, which holds whenever the coefficients $(b_j)_{j \in \mathbb{N}}$ decay sufficiently fast. In this sense, the $r$-approximating property formalizes that the influence of distant $(Z_t)_{t\in \mathbb{Z}}$ becomes negligible, i.e., that $X_t$ depends primarily on a neighborhood of the driving processes $(Z_j)_{j\in \mathbb{Z}}$ at time $t$. The concept of $r$-approximating functional is closely related to the classical notion of $L_r$ near-epoch dependence (NED), and the two concepts are often used interchangeably in the literature. We refer the reader to \cite{schnurr2017testing} for further details and additional examples.

\begin{theorem}[Under  $\mathcal{H}_1$]
\label{thm:alt-limit}
Let $(X_t)_{t \in \mathbb{N}}$ be a $1$-approximating functional with approximating constants $(a_k)_{k \geq d}$ of an absolutely regular time series $(\bm{Z}_t)_{t \in \mathbb{Z}}$ with $\beta$-mixing coefficients $(\beta_k) _{k\geq 0}$. 
Let $F$ denote the marginal distribution of $(\bm{X}_t)_{t \in \mathbb{N}}$, and let $h : \mathbb{R}^{d} \times \mathbb{R}^d \to \mathbb{R}$ be the symmetric kernel defined in \eqref{eq:kernel}. Further, we define $h_1(\bm{x}):=\mathbb{E}[h(\bm{x},\bm{Y})]$ for $\bm{Y}\sim F$ and all $\bm{x}\in \mathbb{R}^d\;.$
If the sequences $(\beta_k) _{k\geq 0}$, $(a_k)_{k \geq d}$, and $(\varphi(a_k)) _{k\geq 0}$ satisfy the summability condition
\[
\sum_{k=d}^\infty k^2 \big( \beta_k + a_k + \varphi(a_k) \big) < \infty,
\]
then the variance series
\[
\sigma^2 = \mathrm{Var}(h_1(\bm{X}_1)) + 2 \sum_{k=1}^\infty \mathrm{Cov}(h_1(\bm{X}_1), h_1(\bm{X}_{1+k}))
\]
converges absolutely, and under the alternative hypothesis  $\mathcal{H}_1$,
\[
\sqrt{n}\,(U_n^d - \theta) \;\xrightarrow{\mathcal{D}}\; \mathcal{N}(0,4\sigma^2),
\]

where
\begin{equation}
\label{eq:theta}
\theta = \mathbb{E}[h(\bm{Y}_1,\bm{Y}_2)] 
= \sum_{G \in \mathcal{G}} \frac{1}{|G|} \Big( \mathbb{P}(\Pi(\bm{Y}_1)\in G) \Big)^2 
- \sum_{\pi \in \mathcal{S}_d} \Big( \mathbb{P}(\Pi(\bm{Y}_1)=\pi) \Big)^2,
\end{equation}
for i.i.d.\ $\bm{Y}_1,\bm{Y}_2$ with distribution $F$.
\end{theorem}

\begin{corollary}
\label{cor:consistency}
Assume that the conditions of Theorem~\ref{theorem:asymptotic_distribution}
hold under $\mathcal H_0$, and that the conditions of
Theorem~\ref{thm:alt-limit} hold under $\mathcal H_1$. Let $c_\alpha$ denote the $\alpha$-quantile of the asymptotic distribution
in Theorem~\ref{theorem:asymptotic_distribution}, and consider the test that
rejects $\mathcal H_0$ whenever
\[
n\,\hat D_2(\mathcal G) < c_\alpha .
\]
Then the test has asymptotic level $\alpha$ and is consistent, i.e.,
\[
\lim_{n\to\infty}
\mathbb P_{\mathcal H_0}\!\left(
n\,\hat D_2(\mathcal G) < c_\alpha
\right) = \alpha,
\qquad
\lim_{n\to\infty}
\mathbb P_{\mathcal H_1}\!\left(
n\,\hat D_2(\mathcal G) < c_\alpha
\right) = 1 .
\]
\end{corollary}
The proof can be found in Appendix \ref{appendix:proofs}. 

\begin{remark}
The assumptions of Theorems~\ref{theorem:asymptotic_distribution}
and~\ref{thm:alt-limit} are stated separately under the null and alternative
hypotheses, respectively. A natural setting in which both sets of assumptions
can be satisfied is given by processes $(X_t)_{t\in\mathbb Z}$ that are
$1$-approximating functionals of a strictly stationary $\beta$-mixing process
whose mixing coefficients satisfy the conditions of
Theorem~\ref{thm:alt-limit}. Such processes are then also $\alpha$-mixing
with coefficients fulfilling the assumptions of
Theorem~\ref{theorem:asymptotic_distribution}.
As a concrete example, an ARMA processes driven by
innovations with an absolutely continuous distribution are absolutely regular
and satisfy $\beta_k=\mathcal O(\rho^k)$ for some $\rho\in(0,1)$, see 
\cite{Mokkadem}. Since geometric decay implies summability of the
$\beta$-mixing coefficients and $\alpha_k\le \beta_k$, this rate is sufficient
to meet the mixing conditions required in both theorems.
\end{remark}

\subsection{Variance Estimation}
\label{sec:var_estimate}
Let $t=d!-m$ and let $\bm{p}=[p_1,\ldots,p_m]^\top$ denote the vector of class probabilities $\mathcal{H}_0$.
Let $\bm{W}=(W_1,\ldots,W_t)^\top$ be the centered Gaussian random vector appearing in
Theorem~\ref{theorem:asymptotic_distribution} with distribution
\[
\bm{W}\sim\mathcal{N}(\bm{0},\boldsymbol{\Sigma}(\bm{p})),
\]
where $\boldsymbol{\Sigma}(\bm{p})$ is the covariance matrix with entries given
in~\eqref{eq:cov_entry}.
Furthermore, let
$\boldsymbol{\Lambda}(\bm{p})
=
\mathrm{diag}(\lambda_1(\bm{p}),\ldots,\lambda_t(\bm{p}))$
collect the nonzero (negative) eigenvalues of the operator \ref{eq:operator}.  
These eigenvalues depend on the class probabilities $\bm{p}$ and satisfy
\begin{align}\label{eq:lambda}
[\lambda_1(\bm{p}), \ldots, \lambda_t(\bm{p})]^\top
=
-\big[p_1, \ldots, p_1,\, p_2, \ldots, p_2,\, \ldots,\, p_m, \ldots, p_m \big]^\top,
\end{align}
where each $p_i$ appears with multiplicity $|G_i|-1$ (see Proposition~\ref{proposition:eigenvalues} in Appendix \ref{appendix:auxiliary_results}). Let $\mathrm{tr}(\cdot)$ denote the trace of a matrix. With this notation in place, the limiting quadratic form in
Theorem~\ref{theorem:asymptotic_distribution} is given by
\begin{equation}
\label{eq:quadratic_form}
S_t(\bm{p})
:=
\sum_{j=1}^t \lambda_j(\bm{p})\,(W_j^2-1)=
\bm{W}^\top \boldsymbol{\Lambda}(\bm{p}) \bm{W}
-
\mathrm{tr}\,\!\big(\boldsymbol{\Lambda}(\bm{p})\big).
\end{equation}
 Hence, the law of $S_t(\bm{p})$ depends on $\bm{p}$. The following result makes this precise. 

\begin{lemma}
\label{lem:variance_dependence}
Let $S_t(\bm{p})$ be the quadratic form defined above and $\varphi_{\bm{p}}(u) := \mathbb{E}[e^{i u S_t(\bm{p})}]$ its characteristic function. If there exist $i,j\in \{1, \ldots, m\}$ with $i\neq j$, such that $|G_i|\neq |G_j|$ then
\[
\bm{p}\mapsto\varphi_{\bm{p}}(u) := \mathbb{E}[e^{i u S_t(\bm{p})}]
\]
is not constant. 
\end{lemma}

\noindent
The proof of Lemma \ref{lem:variance_dependence} can be found in Appendix \ref{appendix:proofs}. According to Lemma \ref{lem:variance_dependence} the null distribution of $\hat{D}_2(\mathcal{G})$ is not parameter-free but generally varies with $\bm{p}$, and consistent estimation of $\bm{p}$ is required to approximate critical values in practice. To this end, note that the estimator $\hat{p}_\mathcal{G}(\pi)$ is  unbiased for $p_\mathcal{G}(\pi)$ under the null and the alternative hypothesis, i.e $\mathbb{E}_{\mathcal{H}_0}[  \hat{p}_\mathcal{G}(\pi)]=\mathbb{E}_{\mathcal{H}_1}[  \hat{p}_\mathcal{G}(\pi)]=p_\mathcal{G}(\pi).$

\medskip

To analyze the covariance structure of $\bm{W}$, Proposition~\ref{proposition:eigenvalues} in Appendix \ref{proposition:eigenvalues} shows that 
the eigenfunctions $g^{(i)}$ can be expressed as linear combinations of indicator functions of ordinal patterns. Specifically,
\[
g^{(i)}(\cdot) = \sum_{s=1}^{d!} g_{\pi_{k_s}}^{(i)}(\bm{p})\, \mathds{1}\big( \Pi(\cdot) = \pi_{k_s} \big), 
\qquad
\bm{g}^{(i)}(\bm{p}) = \big[ g_{\pi_{k_1}}^{(i)}(\bm{p}), \ldots, g_{\pi_{k_{d!}}}^{(i)}(\bm{p})\big]^\top.
\]
From Theorem~\ref{theorem:asymptotic_distribution}, $\bm{W}$ has zero mean, so its covariance matrix satisfies 
\begin{align}
\label{eq:covariance_W}
    \mathbb{E}[W_{i} W_{j}] = \lim_{n\to \infty} \frac{1}{n} \sum_{s,\ell=1}^n \mathbb{E}[g^{(i)}(\bm{X}_s) g^{(j)}(\bm{X}_\ell)], 
    \qquad i,j=1,\ldots,d!-m.
\end{align}
For $\ell \geq 0$, we define the stationary vector process
\[
\bm{R}_{n,\ell} = \frac{1}{\sqrt{n}}\big[ \mathds{1}\big( \Pi(\bm{X}_\ell) = \pi_{k_1} \big), \ldots, 
\mathds{1}\big( \Pi(\bm{X}_\ell) = \pi_{k_{d!}} \big)\big]^\top\;.
\]
Expanding the eigenfunctions gives
\begin{align*}
   \frac{1}{n} \mathbb{E}\big[ g^{(i)}(\bm{X}_s)g^{(j)}(\bm{X}_{\ell}) \big] 
    &=  \frac{1}{n}  \sum_{v=1}^{d!} \sum_{r=1}^{d!} g_{\pi_{k_v}}^{(i)}(\bm{p}) g_{\pi_{k_r}}^{(j)} (\bm{p})
       \,\mathbb{E}\Big[ \mathds{1}\big( \Pi(\bm{X}_s) = \pi_{k_v} \big) 
                         \mathds{1}\big( \Pi(\bm{X}_\ell) = \pi_{k_r} \big) \Big] \\
    &= (\bm{g}^{(i)} (\bm{p}))^\top \,
       \mathbb{E}\!\left[ \bm{R}_{n,s} \bm{R}_{n,\ell}^\top \right] 
       \bm{g}^{(j)}(\bm{p})\;.
\end{align*}
Define
\begin{equation}
\label{eq:omega_def}
\bm{\Omega}_n := \sum_{s,\ell=1}^n \mathbb{E}[ \bm{R}_{n,s} \bm{R}_{n,\ell}^\top], 
\end{equation}
Then, the covariance of the limiting variables satisfies
\[
\mathbb{E}[W_i W_j] = \lim_{n\to \infty} (\bm{g}^{(i)} (\bm{p}) )^\top\, \bm{\Omega}_n \, \bm{g}^{(j)}(\bm{p})\;.
\]
If we can construct an estimator $\hat{\bm{\Omega}}_n$ satisfying  
\[
\hat{\bm{\Omega}}_n - \bm{\Omega}_n \xrightarrow{\mathbb{P}} \bm{0},
\]
and use a consistent estimator $\hat{\bm{p}}_n$ of $\bm{p}$, then by continuity of $\bm{g}_i(\cdot)$, as a consequence of the continuous mapping theorem we obtain
\[
(\bm{g}^{(i)}(\hat{\bm{p}}_n))^\top\hat{\bm{\Omega}}_n \bm{g}^{(j)}(\hat{\bm{p}}_n) 
- (\bm{g}^{(i)}(\bm{p}))^\top\bm{\Omega}_n \bm{g}^{(j)}(\bm{p}) \xrightarrow{\mathbb{P}}0.
\]
This provides an operative way to estimate the entries of the covariance matrix of $\bm{W}$.  In particular, 
\begin{align}
\label{eq:estimation}
     \mathbb{E}[W_i W_j]=& \lim_{n\to \infty} (\bm{g}^{(i)}(\hat{\bm{p}}_n) )^\top\hat{\bm{\Omega}}_n \bm{g}^{(j)}(\hat{\bm{p}}_n) \;.
\end{align}
To construct such an estimator $\hat{\bm{\Omega}}_n$ we adopt a kernel-based long-run covariance estimator \cite{dejong2000consistency}:
\begin{equation}
\label{eq:omega_est}
\hat{\bm{\Omega}}_n = \sum_{s,t=1}^n \bm{R}_{n,s} \bm{R}_{n,\ell}^\top \, k\!\left( \frac{\ell-s}{\gamma_n}\right),
\end{equation}
where $k$ is a kernel function and $\gamma_n$ is a bandwidth parameter satisfying standard regularity conditions. Theorem 2.1 of \cite{dejong2000consistency} guarantees
$ \hat{\bm{\Omega}}_n - \bm{\Omega}_n \xrightarrow{\mathbb{P}}\bm{0}$ for \eqref{eq:omega_est} under the following technical conditions. 
\begin{assumption}
\label{ass:kernel}
Let \( k(\cdot) : \mathbb{R} \to [-1, 1] \) be a kernel function satisfying the following conditions:
\begin{itemize}
    \item \( k(0) = 1 \), \( k(\cdot) \) is symmetric: \( k(x) = k(-x) \) for all \( x \in \mathbb{R} \),
    \item \( k(\cdot) \) is continuous at 0 and at all but a finite number of points,
    \item \( \displaystyle \int_{-\infty}^{\infty} |k(x)|\, dx < \infty \) and \( \displaystyle \int_{-\infty}^{\infty} |\psi(\xi)|\, d\xi < \infty \), where
    \(
    \psi(\xi) = \frac{1}{2\pi} \int_{-\infty}^{\infty} k(x) e^{i \xi x} \, dx
    \)
    is the Fourier transform of \( k \).
\end{itemize}
\end{assumption}
\begin{assumption}
\label{assumption:dejong2-raw}
The array \( \bm{R}_{n,t} \) is 2-approximating functional  of size \( -1 \) on a strong mixing random array \( V_{nt} \) of size \( -r/(r-2) \) (with \( r > 2 \), resp. \( r > 1 \)), and there exists a triangular array of constants \( c_{nt} \) such that
\begin{equation}
\label{eq:dejong2.6}
\sup_{n \geq 1} \sup_{1 \leq t \leq n} \left( \mathbb{E}  \|\bm{R}_{n,t}\|_r +d_{nt} \right)/c_{nt} < \infty,
\end{equation}
for some \( r > 2 \), and
\begin{equation}
\label{eq:dejong2.7}
\sup_{n \geq 1} \sum_{t=1}^n c_{nt}^2 < \infty.
\end{equation}
\end{assumption}

\begin{assumption} The bandwidth $ \gamma_n$ satisfies  
\label{assumption:dejong3-raw}
\[
\lim_{n \to \infty} \left( \frac{1}{\gamma_n} +\gamma_n\max_{1 \leq s \leq n} c_{ns}^2 \right) = 0.
\]
\end{assumption}

\begin{theorem}
\label{thm:omega_consistency}
Let $(X_t)_{t\in\mathbb{Z}}$ be a $2$-approximating functional of size $-4$ of a stationary process $(Z_t)_{t\in\mathbb{Z}}$ of size $-r/(r-2)$ for some $r>2$. 
Assume further that for $0\le i<j\le d-1$  the differences $X_i-X_j$ admit a Lipschitz continuous cumulative distribution function. 
Then, for any kernel $k$ satisfying Assumption~\ref{ass:kernel}, and any bandwidth sequence $(\gamma_n)_{n\geq 1}$ satisfying Assumption~\ref{assumption:dejong3-raw} with $c_{nt}=1/\sqrt{n}$, it holds
\[
\hat{\bm{\Omega}}_n - \bm{\Omega}_n \xrightarrow{\mathbb{P}} \bm{0}.
\]
\end{theorem}
The proof can be found in Appendix \ref{appendix:proofs}.

\section{Simulations}
\label{sec:simulations}
\paragraph*{Setting A (under  $\mathcal{H}_0$)}

\label{ex:simulation}
Let us consider the partition $\mathcal{G} = \{G_1,G_2\}$ with 
\[
G_1=\{(2,3,1),(1,3,2),(3,1,2),(2,1,3)\} \ \ \text{and} \ \
G_2=\{(1,2,3),(3,2,1)\},
\]
corresponding to the case $d=3$ and hence $d! = 6$ ordinal patterns. The data are generated from a $\mathrm{MA}(1)$ process
\begin{align}
\label{eq:MA}
    X_t = \varepsilon_t + \theta \varepsilon_{t-1}, \qquad \varepsilon_t \sim \mathcal{N}(0,1), 
\end{align}
with fixed parameter $\theta=0.5$ and independent Gaussian innovations $( \varepsilon_{t})_{t\geq 1}$. Since the innovations are Gaussian, the entire process $(X_t)_{t\in\mathbb{Z}}$ is stationary Gaussian. 
This entails both spatial and time-reversal symmetries (see for instance \cite{sinn2011estimation}), implying in particular that 
$p((1,2,3)) = p((3,2,1))$ and that the other corresponding pairs of patterns have equal 
probabilities as well. This setting therefore respects  the null hypothesis. 

We illustrate the limiting distribution of $n\hat{D}_2(\mathcal{G})$ according to Theorem \ref{theorem:asymptotic_distribution}. The group $G_1$ has size $4$ and contributes $d_1-1=3$ eigenvectors, while 
$G_2$ has size $2$ and contributes $d_2-1=1$ eigenvector, yielding 
$t = d! - m = 6-2=4$ eigenfunctions in total (see Proposition \ref{proposition:eigenvalues}). Via the isomorphism in \eqref{eq:isomorphism}, these eigenfunctions correspond uniquely to vectors in $\mathbb{R}^{d!}$.
In this example, they are given by:
\[
\mathbf{g}^{(1)}=\frac{1}{\sqrt{2p_1}}\begin{bmatrix}
1 \\ -1 \\ 0 \\ 0 \\ 0 \\ 0
\end{bmatrix}, \qquad
\mathbf{g}^{(2)}=\frac{1}{\sqrt{6p_1}}\begin{bmatrix}
1 \\ 1 \\ -2 \\ 0 \\ 0 \\ 0
\end{bmatrix}, \qquad
\mathbf{g}^{(3)}=\frac{1}{\sqrt{12p_1}}
\begin{bmatrix}
1 \\ 1 \\ 1 \\ -3 \\ 0 \\ 0
\end{bmatrix}, \qquad
\mathbf{g}^{(4)}=\frac{1}{\sqrt{2p_2}}\begin{bmatrix}
0 \\ 0 \\ 0 \\ 0 \\ 1 \\ -1
\end{bmatrix}.
\]
We estimate $p_1$ and $p_2$ by $\hat{p}_1=\hat{p}_n(G_1),\, \hat{p}_2=\hat{p}_n(G_2)$ via \eqref{eq:group_estimator}.
We then compute an estimator for the covariance matrix
\[
\widehat{\mathbb{E}}[W_i W_j] 
\sim \left(g^{(i)}(\hat{\mathbf{p}}_n(\mathcal{G}))\right)^\top \, \hat{\mathbf{\Omega}}_n \, g^{(j)}(\hat{\mathbf{p}}_n(\mathcal{G})),\qquad\hat{\mathbf{p}}_n(\mathcal{G})=[\hat{p}_1,\hat{p}_2],
\]
and simulate from 
\[
S_t(\hat{\mathbf{p}}_n(\mathcal{G}))+c \;=\; -\sum_{i=1}^t \hat{p}_i \,\big(W_i^2 - 1\big)+c.
\]
Figure~\ref{fig:simulated_St} displays the resulting empirical density of both  $n\hat{D}_2(\mathcal{G})$ and  $S_t(\mathbf{p})+c$ 
for a sample of $N=2000$ simulations.
\begin{figure}[h!]
    \centering
    \begin{minipage}{0.45\textwidth}
        \centering
        \includegraphics[width=\textwidth]{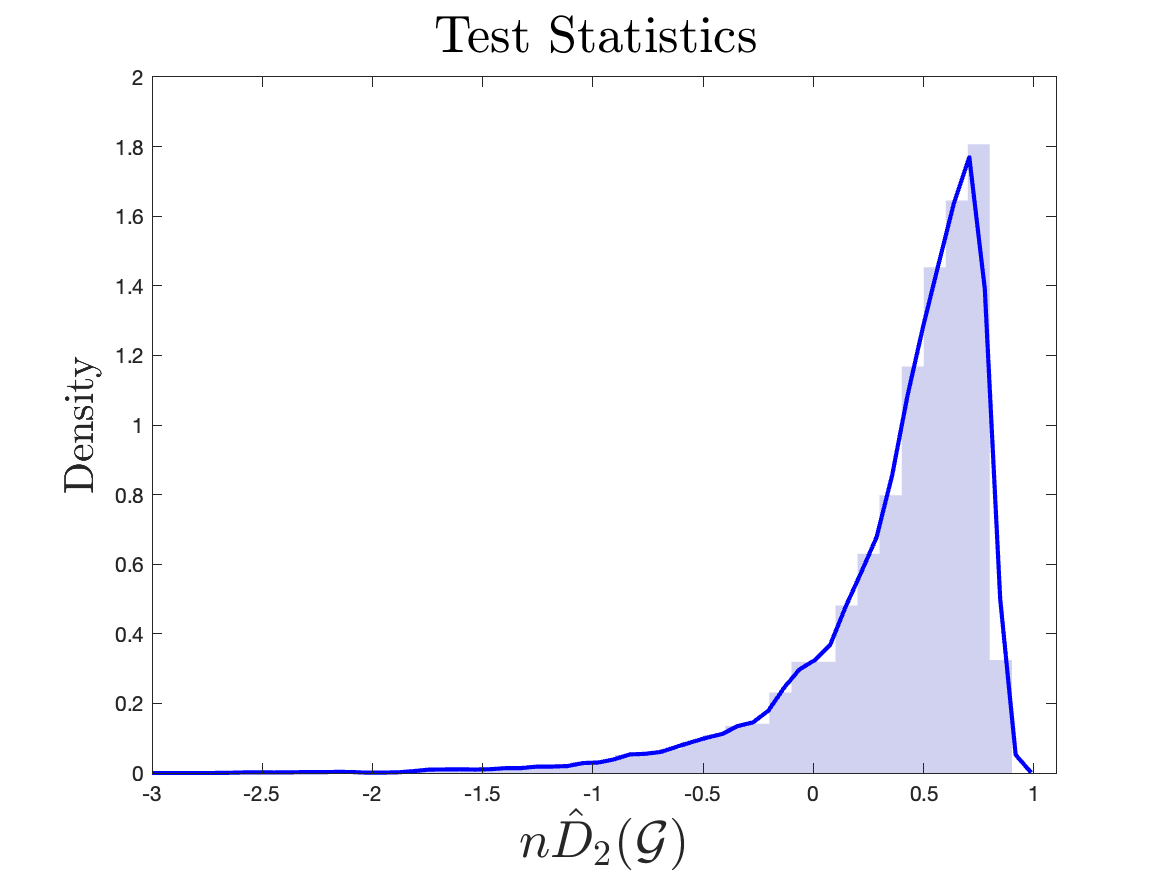} 
    \end{minipage}
    \hfill
    \begin{minipage}{0.45\textwidth}
        \centering
        \includegraphics[width=\textwidth]{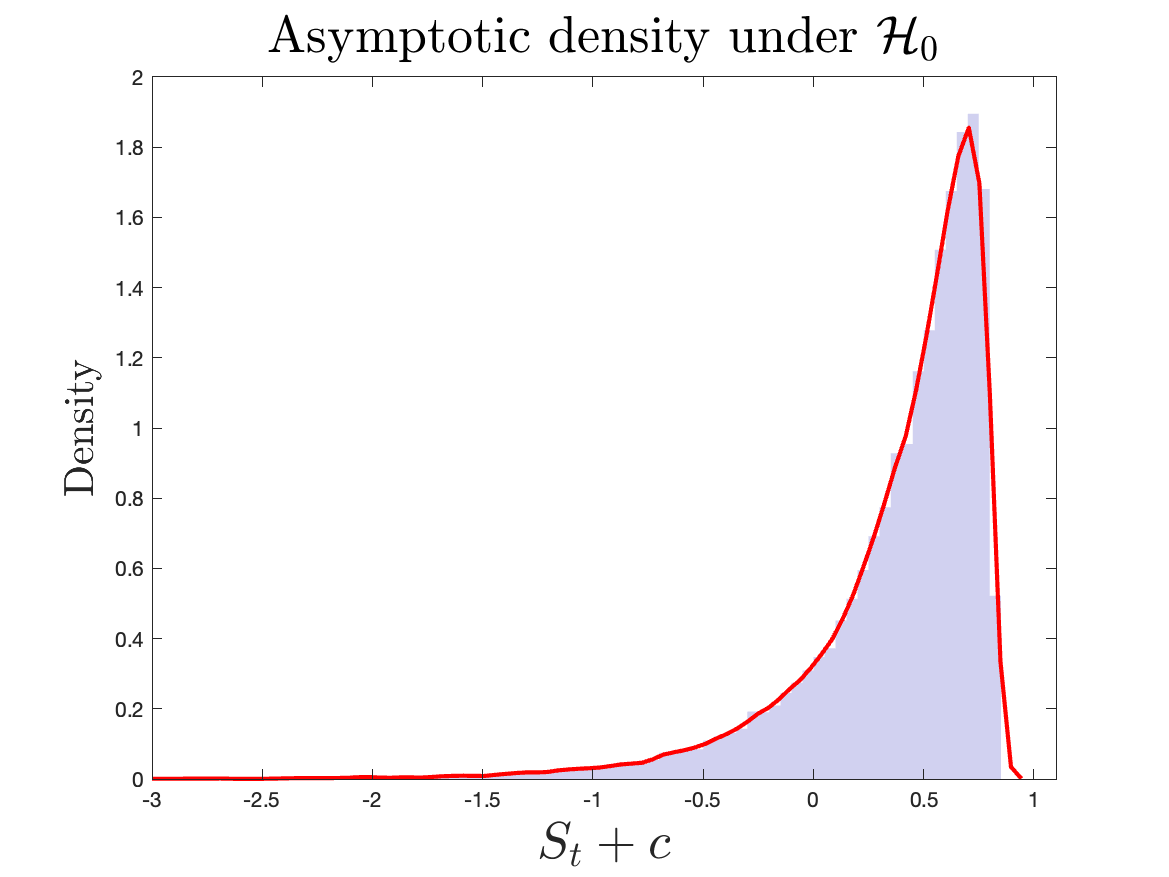} 
    \end{minipage}
     \caption{Simulated distribution of $S_t(\mathbf{p})+c$ under the null hypothesis
             obtained from $N=2000$ Gaussian draws with covariance estimated
             from $\hat{\Omega}_n$, for $n=1000$. 
             The red line shows the kernel density estimate. Left: mean 0.3510, variance 0.2321. Right: mean 0.3604, variance 0.2356}
     \label{fig:simulated_St}
\end{figure}

\paragraph*{Setting B (under  $\mathcal{H}_1$)}
We also provide an illustration of Theorem~\ref{theorem:asymptotic_distribution} 
under the alternative hypothesis. We simulate again an MA(1) process with Gaussian innovations, exactly as in the previous example. Consider the partition 
$\mathcal{G} = \{G_1,G_2\}$ with 
\[
G_1=\{(1,2,3),(1,3,2),(3,1,2),(2,1,3)\},\qquad 
G_2=\{(2,3,1),(3,2,1)\}.
\]
With the above choice of partition, the group $G_2$ 
collects two patterns whose probabilities are not equal under the Gaussian assumption. 
Hence, for Gaussian data $\mathcal{G}$ does not respect the symmetry structure of the process,  such that the described setting falls under  $\mathcal{H}_1$.

\begin{figure}
    \centering
    \begin{minipage}{0.45\textwidth}
        \centering
         \includegraphics[width=\textwidth]{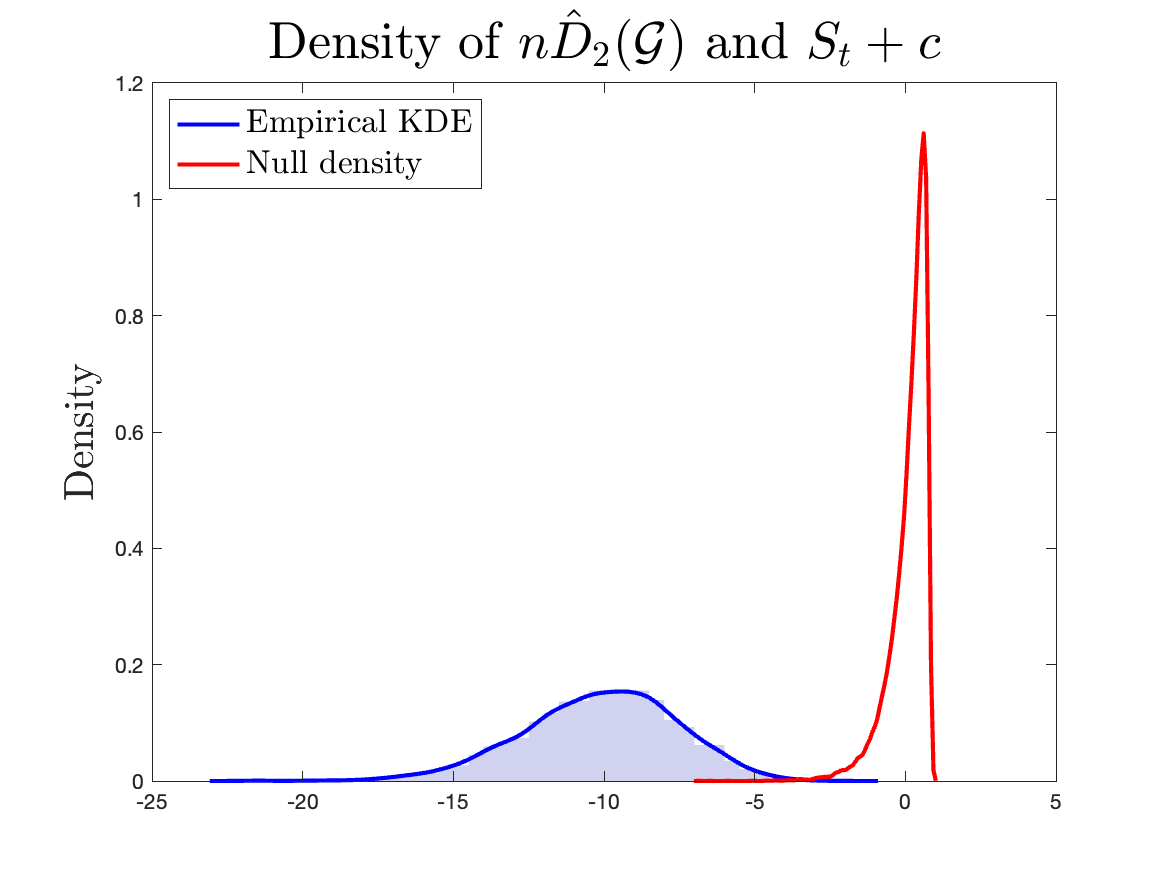} 
    \end{minipage}
    \hfill
    \begin{minipage}{0.45\textwidth}
        \centering
        \includegraphics[width=\textwidth]{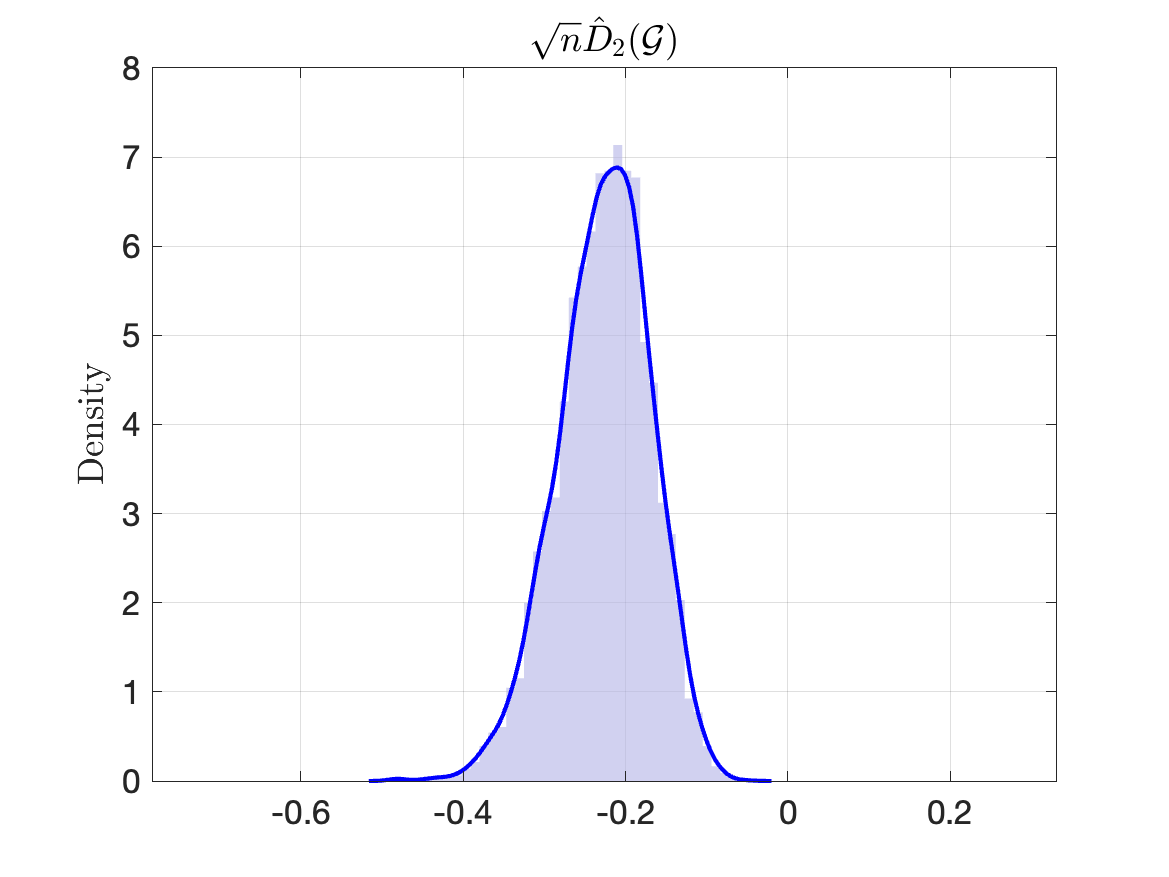} 
    \end{minipage}
     \caption{On the left: simulated values of $n \hat{D}_2(\mathcal{G})$ 
together with the density of $S_t(\mathbf{p})$ under $\mathcal{G}$. 
On the right: simulated values of $\sqrt{n}\,\hat{D}_2(\mathcal{G})$ 
and the corresponding kernel density estimate, which exhibits 
approximately normal behavior. Each histogram is based on 
$N=2000$ independent replications, with time series length $n=1000$.}

     \label{fig:simulated_St_alternative}
\end{figure}
For completeness, we also report the power of the test in the considered Setting B. Specifically, we vary the AR(1) and MA(1) parameters of the underlying processes and evaluate the test’s power. In our simulations, the power is measured as the proportion of rejections of the null hypothesis at the $5\%$ level. Concretely, we generate $N$ independent time series of length $n$, either from
$$
X_t = \theta X_{t-1} + \varepsilon_t \quad \text{(AR(1))}, 
\qquad \text{or} \qquad 
X_t = \varepsilon_t + \theta \varepsilon_{t-1} \quad \text{(MA(1))},
$$
with i.i.d.\ innovations $(\varepsilon_t)_{t\in \mathbb{Z}}$. For each simulated series we compute $n \hat{D}_2(\mathcal{G})$ and reject the null if this statistic falls below the $5\%$ quantile  of its asymptotic distribution. The empirical power is then
$$
\text{Power} = \frac{1}{N} \sum_{j=1}^N 
\mathds{1}\!\left(n \hat{D}_2^{(j)}(\mathcal{G}) < q_{0.05}\right),
$$
where $q_{0.05}$ denotes the $5\%$ quantile of the asymptotic null distribution $S_t(\mathbf{p})+c$.

Note that the limiting distribution of $S_t(\mathbf{p})$ depends on the
vector of grouped ordinal patterns probabilities
$ \mathbf{p}=(p_1, \ldots, p_m)^\top $.
The plug-in estimator $\hat{\mathbf p}_{n}(\mathcal{G})=\left(\hat{p}_n(G_1),\ldots,\hat{p}_n(G_m)\right)^\top$ is consistent under both
 $\mathcal{H}_0$ and  $\mathcal{H}_1$.
Consequently, the same simulation-based procedure can be applied irrespective of whether the null or the alternative holds, which is essential since the true hypothesis is unknown in practice.

While the simulation algorithm is therefore identical, the resulting limiting
distribution differs from the one illustrated in Setting ~A. 
This difference is illustrated in Figure~\ref{fig:simulated_St_alternative}
(left), whose limiting density is distinct from that shown in
Figure~\ref{fig:simulated_St} (right), even though the two densities may appear
visually similar in shape. As shown in Table \ref{tab:power_combined}, the power decreases for small values of \(\theta\). This is expected: when \(|\theta|\) is close to zero, the AR(1) and MA(1) processes exhibit only very weak serial dependence. In this nearly independent regime, the joint distribution of consecutive observations is almost indistinguishable from that of i.i.d. data, and the associated ordinal patterns probabilities are all close to \(1/d!\).

\begin{table}[ht]
\centering
\setlength{\tabcolsep}{4pt} 
\renewcommand{\arraystretch}{1.1} 
\begin{tabular}{l|cccc|cccc}
 & \multicolumn{4}{c|}{AR(1)} & \multicolumn{4}{c}{MA(1)} \\
\hline
$\theta$ & $n=500$ & $n=1000$ & $n=1500$ & $n=2000$ 
         & $n=500$ & $n=1000$ & $n=1500$ & $n=2000$ \\
\hline
0.1 & 0.082 & 0.082 & 0.161 & 0.207  & 0.092 & 0.085 & 0.186 & 0.244 \\
0.3 & 0.024 & 0.237 & 0.903 & 0.994  & 0.089 & 0.624 & 0.988 & 0.999 \\
0.5 & 0.972 & 0.999 & 1.000 & 1.000  & 1.000 & 1.000 & 1.000 & 1.000 \\
\end{tabular}
\caption{Empirical power of the test for AR(1) $X_t=\theta X_{t-1}+\varepsilon_t$  and MA(1)  $X_t=\varepsilon_t+\theta \varepsilon_{t-1}$ with different $\theta$ and sample sizes $n$ (based on $N=1000$ replications). }
\label{tab:power_combined}
\end{table}

\paragraph*{Setting C} 
We consider a process that exhibits non-Gaussian marginal distributions while still satisfying both spatial and temporal reversals, i.e., we test for the partition 
$\mathcal{G} = \{G_1,G_2\}$ with 
\[
G_1=\{(2,3,1),(1,3,2),(3,1,2),(2,1,3)\} \ \ \text{and} \ \
G_2=\{(1,2,3),(3,2,1)\}\;.
\]
A natural way to construct such a process is the following. Let $(Y_t)_{t \in \mathbb{Z}}$ be a stationary Gaussian process, and let $g:\mathbb{R}\to\mathbb{R}$ be one-to-one.  
Then the subordinated process
\[
X_t = g(Y_t) \qquad t \in \mathbb{Z}\;,
\]
inherits all time-reversal and reflection symmetries  of $(Y_t)_{t \in \mathbb{Z}}$  while typically possessing non-Gaussian marginals.
In fact if $g$ is one-to-one and $(Y_t)_{t \in \mathbb{Z}}$ is a time series, then $(Y_t)_{t \in \mathbb{Z}}$
is time reversible if and only if the transformed process $(g(Y_t))_{t \in \mathbb{Z}}$ is time
reversible (see \cite{Weiss_1975}). Hence, strictly monotone transformations preserve (non)reversibility. Moreover, since $(Y_t)_{t\in\mathbb{Z}}$ is Gaussian and centered, its finite-dimensional distributions satisfy the spatial symmetry
\(
(Y_0,\ldots,Y_{d-1})
\;\stackrel{\mathcal{D}}{=}\;
(-Y_0,\ldots,-Y_{d-1})
\). Since equality in distribution is preserved under measurable transformations, applying the transformation $g$ componentwise yields
\[
(X_0,\ldots,X_{d-1})
=
(g(Y_0),\ldots,g(Y_{d-1}))
\;\stackrel{\mathcal{D}}{=}\;
(g(-Y_0),\ldots,g(-Y_{d-1}))\;,
\]
i.e., $(Y_t)_{t\in \mathbb{Z}}$ is symmetric. Therefore, if $Y_t \sim \mathcal{N}(\mu,\sigma^2)$ and we wish to prescribe the marginal law of $X_t$ to be a given distribution~$F$, it suffices to define
\[
g(y) = F^{-1}\!\left(\Phi\!\left(\frac{y-\mu}{\sigma}\right)\right),
\]
where $\Phi$ denotes the standard normal cumulative distribution function and $F^{-1}(y)=
\inf \{ x:\,|\,
F(x) \geq  y\}$ is the quantile function of $F$. By construction, $g$ is monotone, and by Proposition 3.1 of \cite{Viitasaari}, $X_t = g(Y_t)$ has marginal distribution~$F$. Further, when $F$ is invertible so is $g$. 
Next, we generate Gaussian AR(1) and MA(1) 
\[
Y_t = \theta Y_{t-1} + \varepsilon_t
\qquad\text{and}\qquad
Y_t = \varepsilon_t + \theta\,\varepsilon_{t-1},
\]
with parameters $ \theta = 0.5$ and 
where the innovations $(\varepsilon_t)_{t\in \mathbb{Z}} \overset{i.i.d}{\sim} \mathcal{N}(0,1)\;.$ Further we set $X_t = g(Y_t)$ for $g$ chosen to match the marginal distributions reported in Table \ref{tab:power_subordinate_hull}.

Under this construction, the null hypothesis of spatial and temporal symmetry
holds. Consequently, the empirical rejection frequency of the proposed test
provides an estimate of its size and should be close to the nominal significance
level, indicating correct Type~I error control. A visual illustration is given
in Figure~\ref{fig:Pareto}, where we consider a time series with Pareto-distributed
marginals.
The empirical rejection rates for significance level
$\alpha = 0.05$ are reported in Table~\ref{tab:power_subordinate_hull}.

\begin{figure}
    \centering
    \begin{minipage}{0.4\textwidth}
        \centering
         \includegraphics[width=\textwidth]{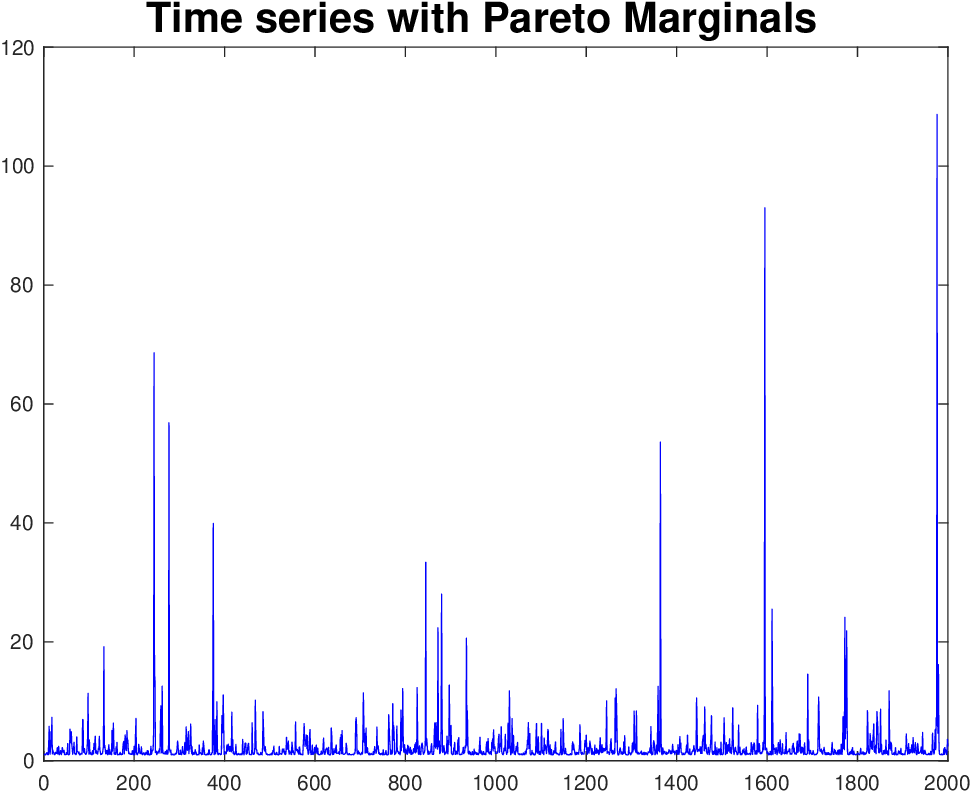} 
    \end{minipage}
    \hfill
    \begin{minipage}{0.5\textwidth}
        \centering
        \includegraphics[width=\textwidth]{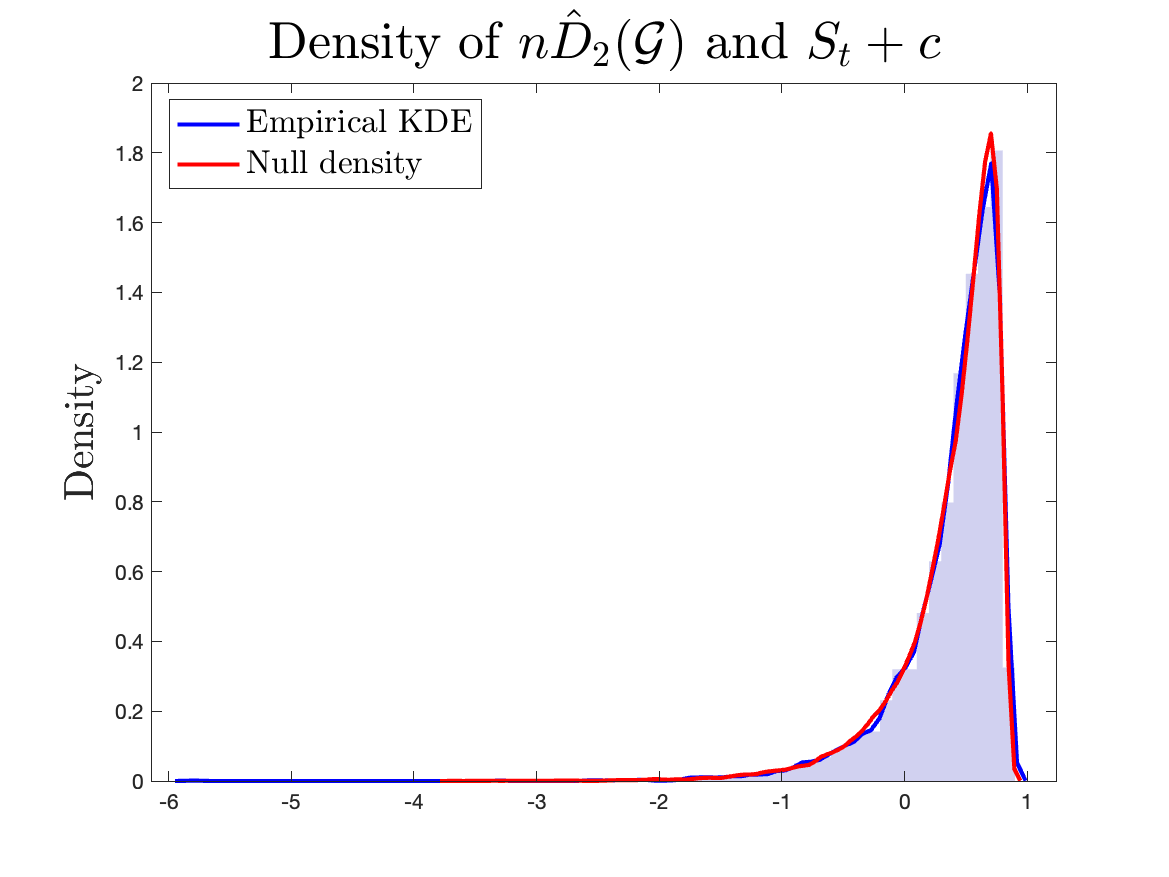} 
    \end{minipage}
    \caption{Left: Example of a time series obtained as a subordinated Gaussian process 
$X_t = g(Y_t)$ with Pareto marginal distribution. 
Right: Density of the asymptotic null distribution together with the empirical distribution of 
$n \hat{D}_2(\mathcal{G})$.}
     \label{fig:Pareto}
\end{figure}
\begin{table}[ht]
\centering
\setlength{\tabcolsep}{4pt} 
\renewcommand{\arraystretch}{1.1} 
\begin{tabular}{l|cccc|cccc}
 & \multicolumn{4}{c|}{\textbf{AR(1)}} & \multicolumn{4}{c}{\textbf{MA(1)}} \\
\hline
$n$& Laplace & Pareto & Logistic & Cauchy 
         & Laplace & Pareto & Logistic & Cauchy \\
\hline
1000 & 0.023 & 0.022 & 0.023 & 0.023  & 0.048 & 0.047 & 0.047 & 0.048 \\
2000& 0.048 & 0.048 & 0.048 & 0.099 & 0.05 & 0.05 & 0.05 & 0.049 \\
\end{tabular}
\caption{Marginal families were parameterized as
logistic $(\mu=100,s=1)$, Pareto $(x_m=1,\alpha=2)$, Laplace $(\mu=1,b=4)$, and Cauchy $(\mu=1,\gamma=12)$.  }
\label{tab:power_subordinate_hull}
\end{table}

It is important to stress that the above construction goes \emph{from} a
Gaussian process $(Y_t)_{t\geq 1}$ \emph{to} a non-Gaussian process $(X_t)_{t\geq 1}$ by
applying a known monotone transformation $g$. The converse direction—
starting with an arbitrary process $(X_t)_{t\geq 1} \sim F$ and attempting to
represent it as $X_t = g(Y_t)$ almost surely for some (not necessarily monotone)
function $g$ and some Gaussian process $(Y_t)_t$—is in general
impossible; see Counterexample~3.1 in \cite{Viitasaari}.
The construction above only guarantees that, for each fixed $t$,
\[
    X_t \;\overset{\mathcal D}{=}\; g(Y_t),
\]
that is, the two processes agree marginally. Such marginal equality does
not imply that $X_t = g(Y_t)$ almost surely, nor that the covariance
structure or any higher-order joint distributions of $(X_t)_{t\geq 1}$ coincide
with those of $(g(Y_t))_{t\geq 1}$. It is precisely these joint distributions,
rather than the one-dimensional marginals, that determine the
ordinal patterns probabilities of a time series.
To see this explicitly, consider an arbitrary stationary process
$(X_t)_{t\in\mathbb{Z}}$ with continuous marginal distribution function $F$.
Let $(Y_t)_{t\in\mathbb{Z}}$ be an i.i.d.\ sequence of standard Gaussian random
variables, and define the monotone transformation
\[
g := F^{-1}\circ \Phi,
\]
where $\Phi$ denotes the standard Gaussian distribution function.
By construction, $g(Y_t)\overset{\mathcal D}{=}X_t$ for each fixed $t$, so that
$(g(Y_t))_{t\in\mathbb{Z}}$ and $(X_t)_{t\in\mathbb{Z}}$ have identical
one-dimensional marginals.
However, since $(Y_t)_{t\in\mathbb{Z}}$ is independent, the transformed process
$(g(Y_t))_{t\in\mathbb{Z}}$ remains i.i.d., and therefore has a uniform
ordinal patterns distribution for any embedding dimension $d$.
In contrast, unless $(X_t)_{t\in\mathbb{Z}}$ itself is independent,
its ordinal patterns distribution is generally non-uniform.

\paragraph*{Setting D} We test the same partition $\mathcal{G}$ of Setting C, namely
$\mathcal{G} = \{G_1,G_2\}$ with 
\[
G_1=\{(2,3,1),(1,3,2),(3,1,2),(2,1,3)\} \ \ \text{and} \ \
G_2=\{(1,2,3),(3,2,1)\}\;.
\]
We simulate linear but non-Gaussian time series. To this end, we generate AR(1) and MA(1) processes of the form 
\[
X_t = \theta X_{t-1} + \varepsilon_t
\qquad\text{and}\qquad
X_t = \varepsilon_t + \theta\,\varepsilon_{t-1},
\]
with parameters $ \theta = 0.5$ and 
where the innovations $(\varepsilon_t)_{t\in \mathbb{Z}}$ have the following marginal distributions
\begin{enumerate}
    \item \textbf{Log-normal:} $\varepsilon_t = \log(\xi_t) - \mathbb{E}[\log(\xi_t)]$, with $\xi_t \sim \mathcal{N}(0,1)$, yielding a positively skewed distribution.
    \item \textbf{Chi-squared:} $\varepsilon_t \sim \chi^2_1 $, producing a heavy right tail.
    \item \textbf{Exponential:} $\varepsilon_t \sim \mathrm{Exp}(1) $, a lighter right-skewed distribution.
    \item \textbf{Student-$t$:} $\varepsilon_t \sim t_1$, exhibiting symmetric but heavy tails.
\end{enumerate}
These models retain linear dependence but violate Gaussianity  due to the non-Gaussianity of the innovations. 
The empirical rejection rates obtained from our test across distributions and sample sizes are summarized in Table~\ref{tab:power_subordinate_alternative}.
These innovations produce stationary processes violating the assumption of joint Gaussianity. 
An empirical illustration is provided in Figure~\ref{fig:histograms}, which displays the
estimated ordinal patterns distributions for the MA(1) process under different innovation
distributions.
\begin{figure}
    \centering
    \begin{tabular}{cc}
        \includegraphics[width=0.45\textwidth]{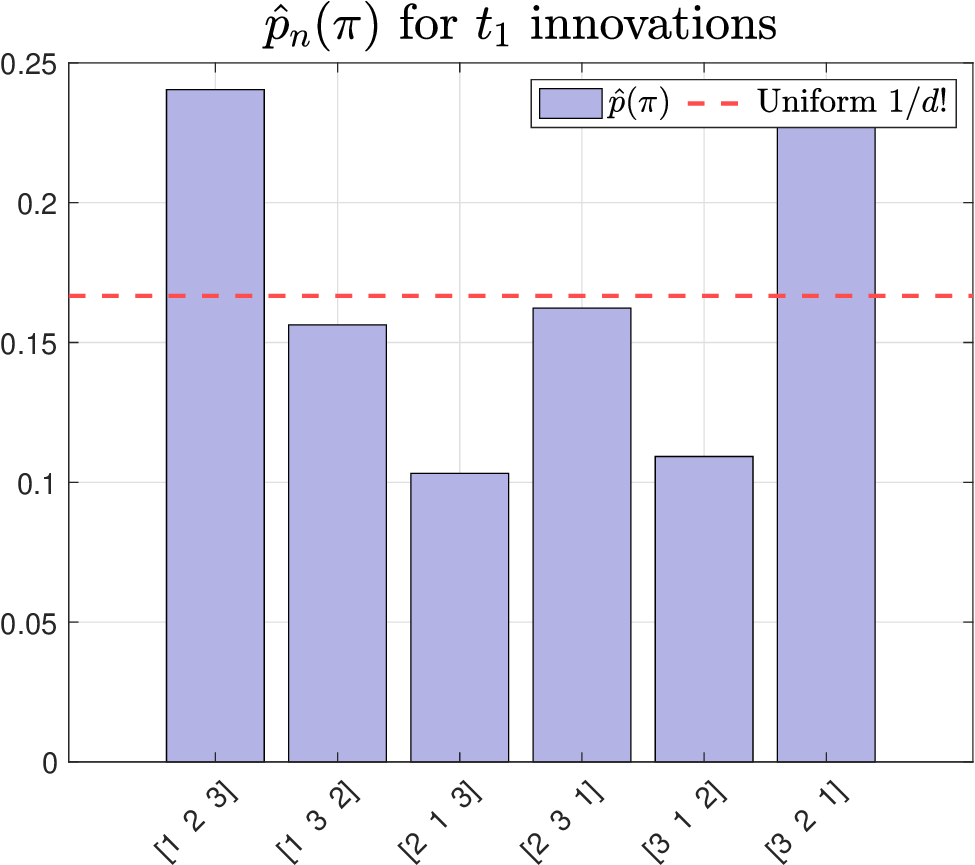} 
        & \includegraphics[width=0.45\textwidth]{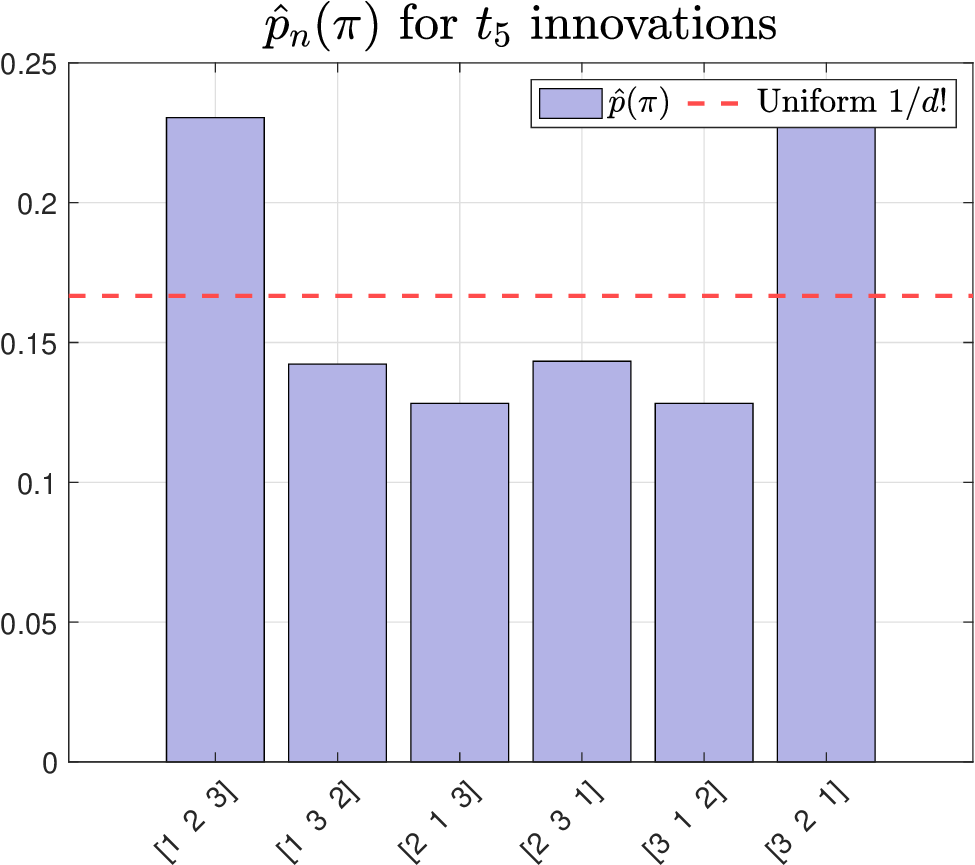} \\[1em]
        \includegraphics[width=0.45\textwidth]{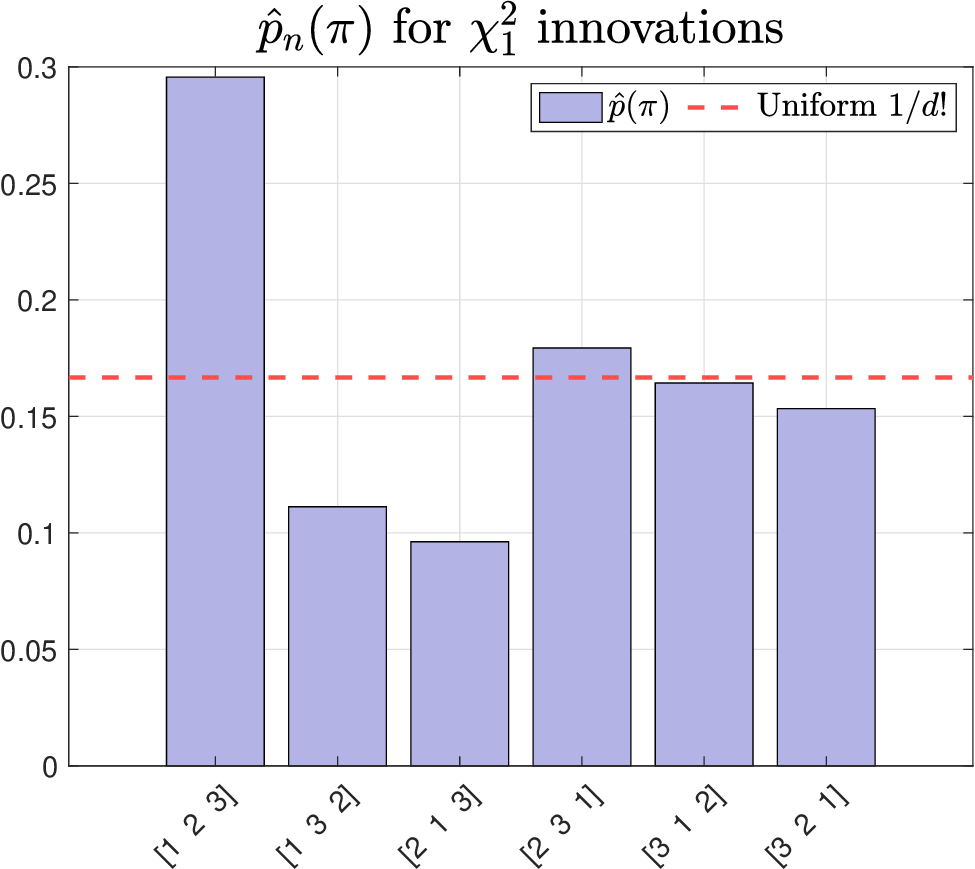} 
        & \includegraphics[width=0.45\textwidth]{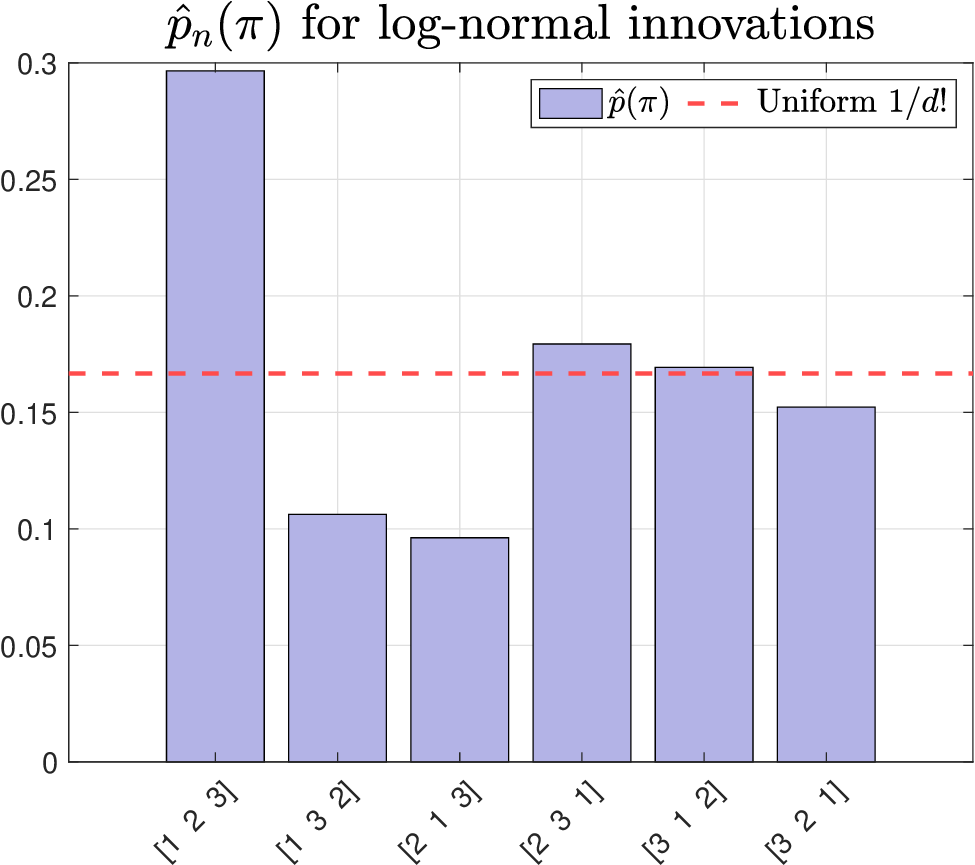} \\[1em]
        \includegraphics[width=0.45\textwidth]{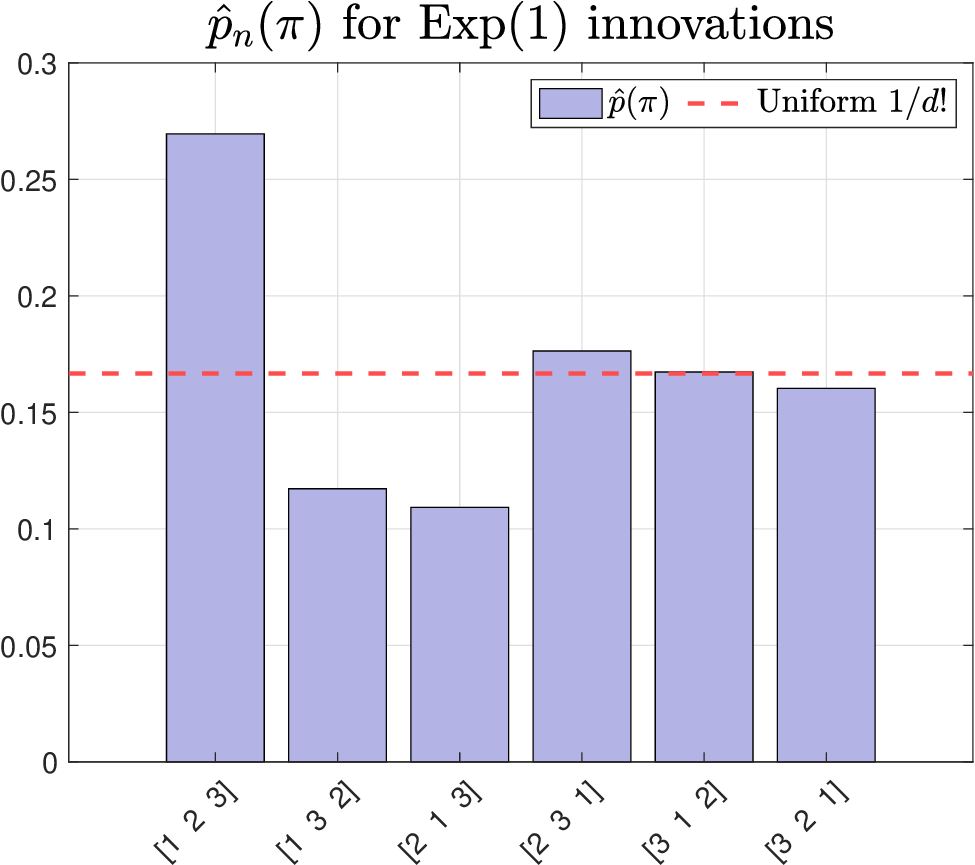} 
        & \includegraphics[width=0.45\textwidth]{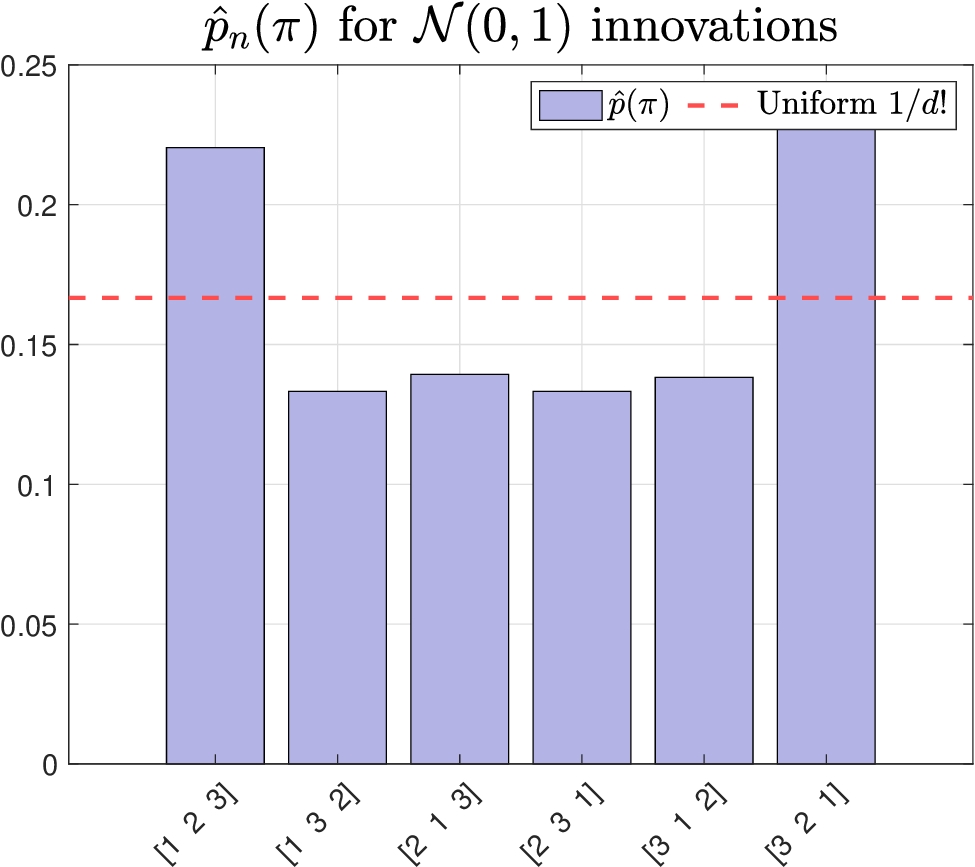}
    \end{tabular}
    \caption{Empirical ordinal patterns distributions for an MA(1) process with different innovation distributions.
Each panel corresponds to a different choice of innovations. The similarity across some panels reflects the fact that the considered innovation distributions are themselves quite similar in shape, leading to nearly indistinguishable ordinal patterns distributions.}
    \label{fig:histograms}
\end{figure}
As illustrated in Figure~\ref{fig:histograms}, the null hypothesis of equal ordinal patterns
probabilities within the prescribed class partition is expected to be rejected with high
probability for most of the considered models. The power of the test, however, will
substantially be reduced for $t$-distributed innovations. The empirical rejection rates obtained
for each model and sample size are reported in Table~\ref{tab:power_subordinate_alternative}.

\begin{table}[ht]
\centering
\setlength{\tabcolsep}{4pt}
\renewcommand{\arraystretch}{1.1}
\begin{tabular}{l|cccc|cccc}
 & \multicolumn{4}{c|}{\textbf{AR(1)}} & \multicolumn{4}{c}{\textbf{MA(1)}} \\
\hline
$n$ & Log-normal & $\chi^2$ & Exp& Student-$t$ 
    & Log-normal & $\chi^2$ & Exp & Student-$t$ \\
\hline
250 & 0.976& 0.997  & 0.788 & 0.083 &  0.765&  0.857 &0.777 & 0.334 \\
500 & 1.000 &  1.000& 0.996 & 0.207 & 0.996 & 0.998 & 0.916 &  0.529\\
1000 & 1.000&  1.000& 1.000 & 0.532 &  1.000&  1.000  &0.998 &0.842  \\

\end{tabular}
\caption{
Empirical rejection rates.  
}

\label{tab:power_subordinate_alternative}
\end{table}

The results in Table~\ref{tab:power_subordinate_alternative} reveal that the test exhibits almost full power for skewed or heavy-tailed innovations such as log-normal, chi-squared, and exponential, since these break both marginal and joint symmetry. 
In contrast, the power is substantially lower for Student-$t$ innovations. 
Although the $t_1$ distribution is non-Gaussian, it is symmetric about zero, and the ordinal pattern probabilities of the resulting AR(1) or MA(1) processes are harder to distinguish from those of Gaussian models.

\subsection{Real Dataset}
\paragraph*{S\&P}
We apply our method to a real-world dataset. Following ~\cite{martinez2018time}, we consider the daily closing values of  Standard \& Poor’s 500 index (S\&P 500) spanning from January 1990 to August 2011, publicly available on several financial platforms (e.g. Yahoo Finance). The dataset consists of ($n=5444$) observations, shown in Figure~\ref{fig:SP}.

As in ~\cite{martinez2018time}, we assess time reversibility using ordinal patterns of size 3. Since the raw S\&P 500 index is well known to exhibit pronounced non-stationarity, we conduct the analysis on the logarithmic returns of the index. Let $\mathcal{G}=\{G_1, G_2, G_3\}$, where
$$ G_1=\{(1,2,3),(3,2,1)\},\quad G_2=\{(1,3,2),(2,3,1)\},\quad G_3=\{(2,1,3),(3,1,2)\}\;.$$
Each pattern is paired with its time-reversed counterpart, yielding the partition. Our procedure tests equality of the probabilities of time-reversed ordinal patterns pairs. Consequently, rejection of the null hypothesis implies that the probabilities of at least one reversed pair differ, which provides evidence of temporal irreversibility of the underlying process.

The resulting $p$-value of our test is of the order $10^{-4}$, indicating a significant rejection of the null hypothesis. 
This finding aligns with those of ~\cite{martinez2018time}, reinforcing the conclusion  that the S\&P 500 series exhibits nonlinear temporal asymmetry. 

\begin{figure}[h!]
    \centering
    \begin{minipage}{0.45\textwidth}
        \centering
         \includegraphics[width=\textwidth]{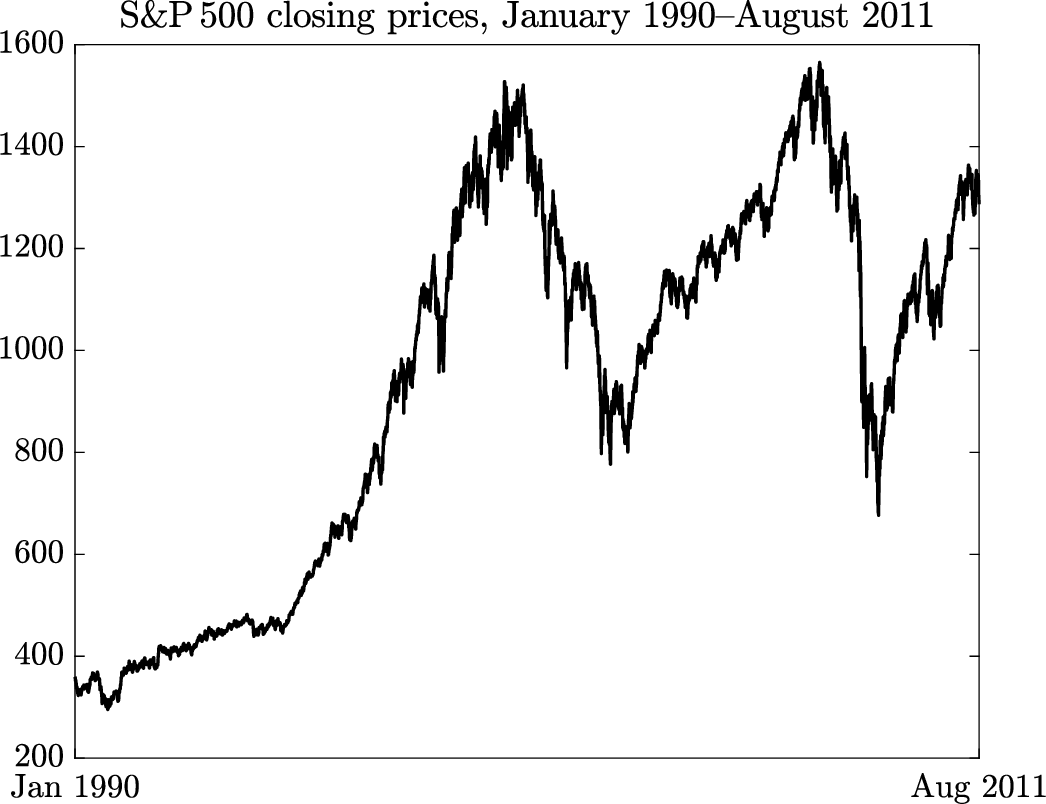} 
    \end{minipage}
    \hfill
    \begin{minipage}{0.45\textwidth}
        \centering
        \includegraphics[width=\textwidth]{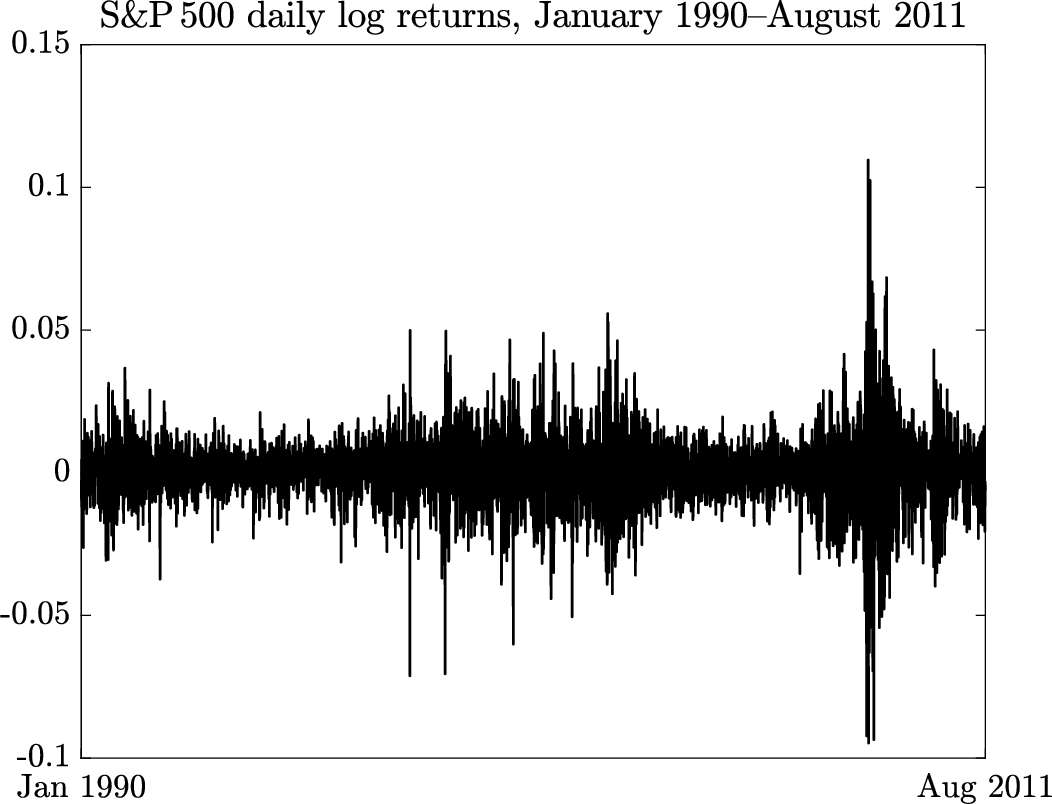} 
    \end{minipage}
  \caption{
Daily closing values of the S\&P\,500 index from January~1990 to August~2011.
Logarithmic returns computed as $\log(X_{t+1})-\log(X_t)$ over the same sample period.
The plot highlights short-term volatility fluctuations around a stable mean. }
     \label{fig:SP}
\end{figure}

\paragraph*{RR Intervals of Heart Rate}

We validate the sensitivity of our symmetry test using heart-rate-variability (HRV) data obtained from RR-interval time series.
The RR interval is the time between two consecutive R-peaks in the electrocardiogram (ECG), that is, between two heartbeats.
These natural variations are not perfectly regular or symmetric, and their departures from time symmetry carry information about the adaptability and healthy functioning of cardiac control.

Previous studies have shown that time irreversibility decreases with age and pathology.
For instance, \cite{Costa} demonstrate that healthy young subjects exhibit the strongest temporal asymmetry in RR-interval sequences, whereas this asymmetry markedly declines in elderly subjects and virtually disappears in patients with congestive heart failure or atrial fibrillation.
This observation has been repeatedly confirmed in later works (we refer the reader to \cite{Li} for an overview of such line of works) indicating that loss of time-irreversible components accompanies both aging and cardiovascular dysfunction.

Our proposed ordinal patterns based framework captures the same physiological phenomenon previously observed in heart-rate dynamics. 
We analyze the publicly available RR Interval Time Series from Healthy Subjects dataset hosted on PhysioNet. The dataset was released by Irurzun et al. \cite{Irurzun2021} and is based on recordings originally collected and analyzed in \cite{Garavaglia}. Access to the data is provided through the PhysioNet repository, \cite{Goldberger2000}. 
The dataset contains 24-hour Holter ECG recordings from healthy volunteers of different ages. 
Each record provides a sequence of RR intervals (time between two consecutive R-peaks in the ECG) expressed in milliseconds, describing the beat-to-beat variability of the heart rhythm. 

We consider three healthy male subjects aged 17, 32, and 53 years, and analyze their corresponding RR-interval time series, see Figure~\ref{fig:rrseries}. 
For each series, we divide the recordings into consecutive non-overlapping blocks of 1500 beats, perform our test for $\mathcal{G}=\{G_1,G_2,G_3\} $ as in the previous example on each block, and record the proportion of rejections of the null hypothesis, as well as the corresponding $p$-value. A block of 1500 observations corresponds to approximately 30 minutes of recording time. Consequently, for each individual, around 40-50 time series per individual are tested. 
We recall that rejection of $\mathcal{H}_0$ indicates that the probabilities of the reversed ordinal patterns classes differ, and hence provides evidence of temporal irreversibility.
The observed age-dependent decline in temporal asymmetry is consistent with previously reported physiological evidence, i.e., that younger individuals exhibit stronger time irreversibility in cardiac dynamics than older ones.
Nevertheless, the dataset we considered is relatively limited: only twelve subjects in the collection are older than 13 years, which constrains our ability to perform more extensive analyses or age-stratified comparisons.
The present results should therefore only be interpreted as a proof of concept. 

\begin{figure}[ht]
    \centering
    \includegraphics[width=0.85\textwidth]{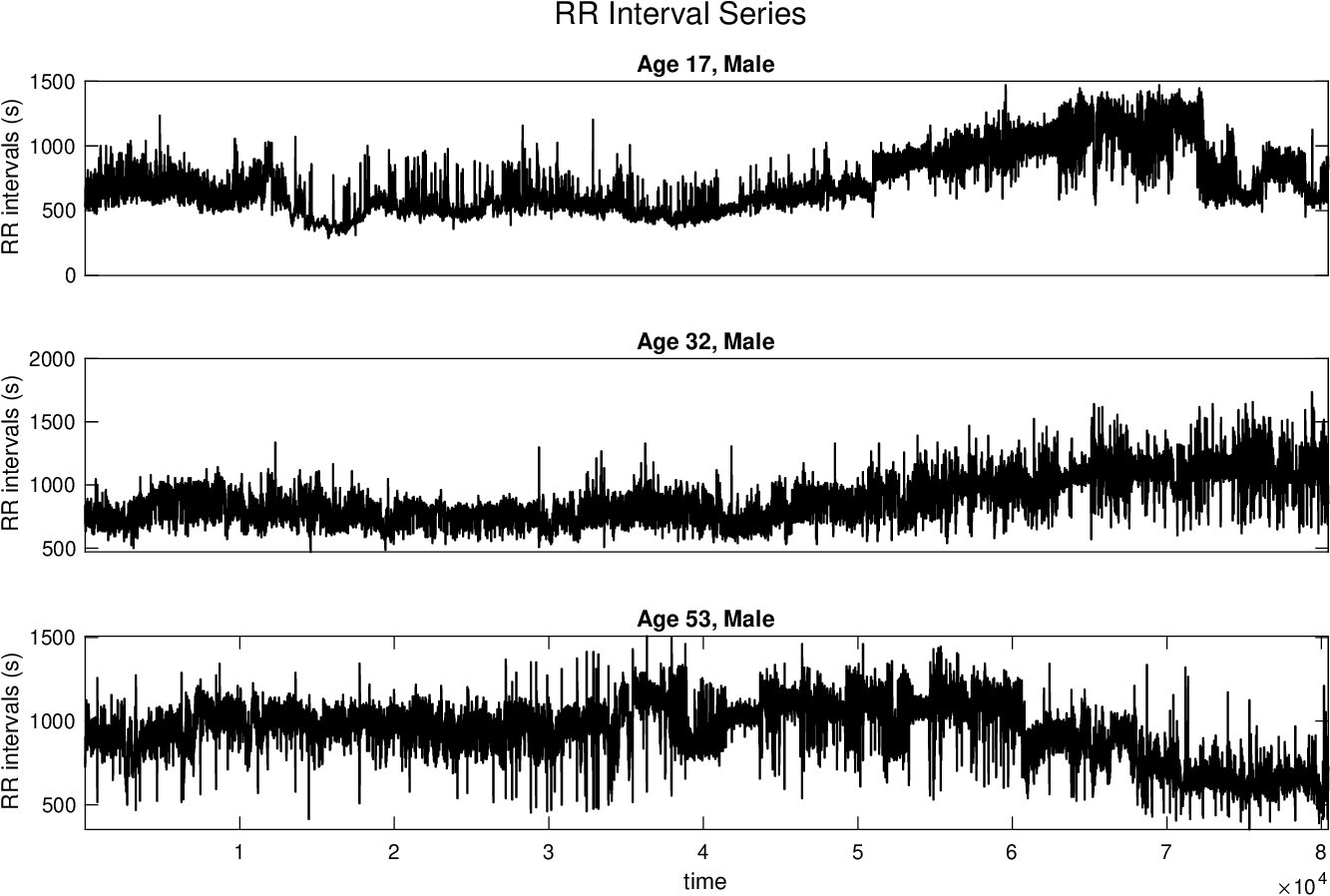}
    \caption{RR-interval time series from three healthy male subjects of ages 17, 32, and 53 years, respectively, extracted from the PhysioNet \emph{RR Interval Time Series from Healthy Subjects} dataset~\cite{Irurzun2021}. 
    Each trace represents the evolution of consecutive heartbeat intervals (in seconds) over time.}
    \label{fig:rrseries}
\end{figure}

\begin{table}[ht]
\centering
\label{tab:hrv_results}
\begin{tabular}{lcc|lcc}
\textbf{Age.y} & \textbf{Aver. p-value} & \textbf{Rej. rate (\%)} & \textbf{Age.y} & \textbf{Aver. p-value} & \textbf{Rej. rate (\%)} \\
\hline
17 M& 0.0390 & 89.0 & 12 F &  0.0268& 86.1\\
32 M& 0.0853 & 71.2 & 24 F& 0.0242  & 93.2 \\
53 M& 0.1379 & 47.2 & 51 F& 0.0026 & 85.1\\

\end{tabular}
\caption{Results of the reversibility test on RR-interval time series. 
The table reports the average p-value and the percentage of rejections of the null hypothesis of temporal symmetry at the 5\% significance level. Female group: patient numbers (413,9,13). Male Group: patient numbers (0, 6, 11).
}
\end{table}

\section{Outlook and Discussion}
\label{sec:conclusion}
In this article, we proved that the ordinal patterns framework provides an interpretable, noise-robust, and model-free tool for testing symmetries in time series. The method is particularly suitable for assessing whether a process originates from a Gaussian model, where deviations from symmetry can reveal nonlinear or non-Gaussian structures.

Future work could aim to establish theoretical guarantees under broader forms of temporal dependence--most notably long-range dependence--to better characterize the asymptotic behavior of ordinal patterns statistics beyond the short-memory setting.

Lastly, the techniques considered to derive the asymptotic distribution of the considered statistic may  go beyond an analysis for ordinal patterns distributions. 
One could, for example, envision analogous results for a goodness-of-fit test for clusters. For inspiration and potential application see \cite{thalamuthu2006evaluation}.
\paragraph*{Data availability statement}
The MATLAB code used to generate the simulations in this article is publicly available at \url{https://github.com/george24GM/Symmetric-patterns.git}

\paragraph*{AI-Assisted Writing Disclosure}
The authors used generative AI-based tools to revise the
text, improve flow and correct any typos, grammatical errors, and
awkward phrasing. 

\paragraph*{Conflict of interest}

The authors declare no potential conflict of interests.

\paragraph*{Financial disclosure}
Annika Betken gratefully acknowledges financial support from the Dutch Research Council (NWO) through VENI Grant~212.164.  \\
Manuel Ruiz Mar\'{\i}n is grateful to the support of a grant from the Spanish Minsitry of Science, Innovation and Universities PID2022-136252NB-I00 founded by MICIU/AEI/10.13039/501100011033 and by the European Regional Development Fund (FEDER, EU).


\nocite{*}
\bibliographystyle{apalike}
\bibliography{literature}



\appendix
\begin{center}
    \textbf{\Large{Appendix}}
\end{center}
This appendix is divided into two parts.
Appendix \ref{appendix:auxiliary_results} collects the auxiliary lemmas and propositions needed to establish the structural properties of the operator  \ref{eq:operator}. Appendix \ref{appendix:proofs} contains the proofs of the main theorems stated in the body of the paper. These proofs build on the results of Appendix \ref{appendix:auxiliary_results}. 

\section{Auxiliary Results}
\label{appendix:auxiliary_results}
The following corresponds to Definition 4.1 of \cite{Schnurr15092025}, and differs slightly from Definition \ref{def:p-continuity}, in the sense that it is specialized for Kernel functions.
\begin{definition}
Let $(Y_t)_{t \in \mathbb{Z}}$ be a stationary time series and \( h : \mathbb{R}^{2d} \to \mathbb{R} \) be a measurable, symmetric kernel. The kernel \( h \) is said to be \( p \)-continuous if there exists a function \( \varphi : (0, \infty) \to (0, \infty) \) with \( \varphi(\epsilon) = o(1) \) as \( \epsilon \to 0 \) such that
\[
\mathbb{E} \left[ |h(U, V) - h(\tilde{U}, V)|^p \cdot \mathds{1} \left( \{\|U - \tilde{U}\| \leq \epsilon \} \right) \right] \leq \varphi(\epsilon),
\]
for all random variables \( U, \tilde{U}, V \) with marginal distribution \( F \) and where \( (U, V) \) either has distribution \( F \times F \) (independent case) or \( \mathbb{P}_{(Y_0, Y_t)} \) for some \( t \in \mathbb{N} \).
\end{definition}

\begin{lemma}\label{lem:lemma_2}
\label{lem:sum_of_continuous_kernels}
Let $h_1, h_2 : \mathbb{R}^{2d} \to \mathbb{R}$ be $p$-continuous kernels with associated functions $\varphi_1$ and $\varphi_2$, respectively. Then, the sum $h = h_1 + h_2$ is also $p$-continuous, with an associated function $\varphi(\epsilon) = 2^{p-1}(\varphi_1(\epsilon) + \varphi_2(\epsilon))$.
\end{lemma}

\begin{proof}
For the sum \( h = h_1 + h_2 \), we have
\[
h(U, V) - h(\tilde{U}, V) = (h_1(U, V) - h_1(\tilde{U}, V)) + (h_2(U, V) - h_2(\tilde{U}, V)).
\]
The Minkowski inequality yields
\[
|h(U, V) - h(\tilde{U}, V)|^p \leq 2^{p-1}\left( |h_1(U, V) - h_1(\tilde{U}, V)|^p + |h_2(U, V) - h_2(\tilde{U}, V)|^p \right).
\]
It follows that
\begin{align*}
    \mathbb{E}\!\left[ |h(U,V)-h(\tilde U,V)|^p 
        \,\mathds{1}\{\|U-\tilde U\|\le \epsilon\} \right] 
    &\le 2^{p-1} \Big(
        \mathbb{E}\!\left[ |h_1(U,V)-h_1(\tilde U,V)|^p 
            \,\mathds{1}\{\|U-\tilde U\|\le \epsilon\} \right] \\
    &\qquad\quad
        + \mathbb{E}\!\left[ |h_2(U,V)-h_2(\tilde U,V)|^p 
            \,\mathds{1}\{\|U-\tilde U\|\le \epsilon\} \right]
        \Big).
\end{align*}

Due to \( p \)-continuity of \( h_1 \) and \( h_2 \)
\[
\mathbb{E}\left[ |h(U, V) - h(\tilde{U}, V)|^p \cdot \mathds{1} \left( \{\|U - \tilde{U}\| \leq \epsilon \} \right) \right] \leq 2^{p-1} \left( \varphi_1(\epsilon) + \varphi_2(\epsilon) \right).
\]
Define \( \varphi(\epsilon) = 2^{p-1}(\varphi_1(\epsilon) + \varphi_2(\epsilon)) \). Since \( \varphi_1(\epsilon) \to 0 \) and \( \varphi_2(\epsilon) \to 0 \) as \( \epsilon \to 0 \), it follows that \( \varphi(\epsilon) \to 0 \). Hence, \( h \) is \( p \)-continuous.
\end{proof}
\begin{proposition}\label{prop:proposition_4}
\label{prop:p-continuity_indicators}
For any $\pi_1, \pi_2 \in \mathcal{S}_d$, the kernel $h(\mathbf{x}, \mathbf{y}) = \mathds{1}\left(\{\Pi(\mathbf{x}) = \pi_1, \Pi(\mathbf{y}) = \pi_2\}\right)$ is \( p \)-continuous with respect to the time series \( (\mathbf{X}_t)_{t \in \mathbb{Z}} \).
\end{proposition}
We follow the argument establishing $p$-continuity of the kernel 
$h(\mathbf{x}, \mathbf{y}) = \mathds{1}\left( \{\Pi(\mathbf{x}) = \Pi(\mathbf{y})\}\right)$ with  respect to the time series \( (\mathbf{X}_t)_{t \in \mathbb{Z}} \) through Proposition 4.2 in \cite{Schnurr15092025}.
\begin{proof}
By our assumptions and the continuity of  \( F \), we have for \( t \in \mathbb{Z} \)
\[
\mathbb{P}(\text{ms}(\mathbf{X}_t) \leq \epsilon) \to \mathbb{P}(\text{ms}(\mathbf{X}_t) = 0) = 0, \quad \text{as } \epsilon \to 0,
\]
where \( \text{ms}(x) := \min \{ |x_j - x_k| : 1 \leq j < k \leq d \} \). Then, for all $\varepsilon>0$
\begin{align*}
&\mathbb{E} \left[ \left| \mathds{1}\left(\{\Pi(U) = \pi_1, \Pi(V)=\pi_2\}\right) - \mathds{1}\left(\{\Pi(\tilde{U})=\pi_1, \Pi(V)=\pi_2\}\right) \right|^p \cdot \mathds{1}\left( \{\|U - \tilde{U}\|_1 \leq \epsilon\} \right) \right]\\
\leq &2^p\int_{\text{ms}(U)\leq 2\epsilon} \mathds{1} \left( \{\|U - \tilde{U}\|_1 \leq \epsilon \} \right)dF\\ +& \int_{\text{ms}(U)> 2\epsilon}
\left|  \mathds{1}\left(\{\Pi(U) = \pi_1, \Pi(V)=\pi_2\}\right) - \mathds{1}\left(\{\Pi(\tilde{U})=\pi_1, \Pi(V)=\pi_2\}\right) \right|
\mathds{1} \left( \{\|U - \tilde{U}\|_1 \leq \epsilon \} \right)
dF\;.
\end{align*}

Note that $\|U - \tilde{U}\|_1 \leq \epsilon$ implies $|U^{(j)}-\tilde{U}^{(j)}|\leq \epsilon$ for every $1\leq j\leq d$. At the same time, if $\text{ms}(U)> 2\epsilon$
it  follows that $|U^{(i)}-U^{(j)}|>2\epsilon$ for any $i\neq j$; see Figure \ref{fig:U's}.

\begin{figure}
    \centering
\begin{tikzpicture}
\draw[thick, ->] (-1, 0) -- (7, 0);

\draw[thick] (0, 0) -- (2, 0);
\draw[thick] (2, 0) -- (4, 0);
\draw[thick] (4, 0) -- (6, 0);

\draw[thick] (0.5, 0.2) arc[start angle=90, end angle=270, radius=0.2];
\draw[thick] (2.5, 0.2) arc[start angle=90, end angle=270, radius=0.2];
\draw[thick] (3.9, 0.2) arc[start angle=90, end angle=270, radius=0.2];

\draw[thick] (1.5, 0.2) arc[start angle=90, end angle=-90, radius=0.2];
\draw[thick] (3.5, 0.2) arc[start angle=90, end angle=-90, radius=0.2];
\draw[thick] (4.9, 0.2) arc[start angle=90, end angle=-90, radius=0.2];

\filldraw[black] (1, 0) circle (2pt) node[below] {$U^{(1)}$};
\filldraw[black] (3, 0) circle (2pt) node[below] {$U^{(3)}$};
\filldraw[black] (4.4, 0) circle (2pt) node[below] {$U^{(2)}$};

\filldraw[red] (0.5, 0) circle (2pt) node[above] {$\tilde{U}^{(1)}$};
\filldraw[red] (2.5, 0) circle (2pt) node[above] {$\tilde{U}^{(3)}$};
\filldraw[red] (4.7, 0) circle (2pt) node[above] {$\tilde{U}^{(2)}$};
\node[above] at (1.3, 0.0) {$\epsilon$};
\node[above] at (3.3, 0) {$\epsilon$};
\node[above] at (4.1, 0) {$\epsilon$};
\end{tikzpicture}
\caption{Illustration of ms$(U)>2\varepsilon$. }
    \label{fig:U's}
\end{figure}
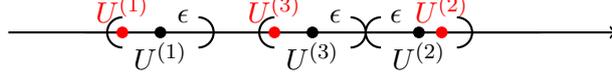
Accordingly, when \( \text{ms}(U) > 2\epsilon \), the condition \( \|U - \tilde{U}\|_1 \leq \epsilon \) ensures that the ordinal patterns \( \Pi(U) \) and \( \Pi(\tilde{U}) \) must be identical. Hence:
\[
\mathds{1}\left(\{\Pi(U) = \pi_1, \Pi(V)=\pi_2\}\right) =\mathds{1} \left(\{\Pi(\tilde{U})=\pi_1, \Pi(V)=\pi_2\}\right)
\]
and the second summand on the right-hand side of the above inequality vanishes.
It therefore follows that 
\begin{align*}
&\mathbb{E} \left[ \left| \mathds{1}\left(\{\Pi(U) = \pi_1, \Pi(V)=\pi_2\}\right) - \mathds{1}\left(\{\Pi(\tilde{U})=\pi_1, \Pi(V)=\pi_2\}\right) \right|^p \cdot \mathds{1}\left( \{\|U - \tilde{U}\|_1 \leq \epsilon\} \right) \right]\\&\leq 2^p \mathbb{P}(\text{ms}(U) \leq 2\epsilon)\;.
\end{align*}
Define \( \varphi(\epsilon) = \mathbb{P}(\text{ms}(U) \leq 2\epsilon) \). Since \( \mathbb{P}(\text{ms}(U) \leq 2\epsilon) \to 0 \) as \( \epsilon \to 0 \), we have \( \varphi(\epsilon) = o(1) \). Therefore, $h(\mathbf{x}, \mathbf{y}) = \mathds{1}\left(\{\Pi(\mathbf{x}) = \pi_1, \Pi(\mathbf{y}) = \pi_2\}\right)$ is \( p \)-continuous.
\end{proof}
The combination of Lemma \ref{lem:lemma_2}  and Proposition \ref{prop:proposition_4} yields that 
$$h(\mathbf{x}, \mathbf{y}) 
= \sum_{G \in \mathcal{G}} \frac{1}{|G|} 
   \mathds{1}\{\Pi(\mathbf{x}) \in G, \, \Pi(\mathbf{y}) \in G\} 
   - \mathds{1}\{\Pi(\mathbf{x}) = \Pi(\mathbf{y})\}$$
   is $p-$continuous. 
\begin{lemma}
\label{lem:ordinal_vector_2approx}
Let $(X_t)_{t\in\mathbb{Z}}$ be a $2$-approximating functional of a stationary process $(Z_t)_{t\in\mathbb{Z}}$ with approximating constants $(a_k)_{k\geq 0}$ of size $-\lambda$, in the sense of Definition \ref{def:r-approx}. 
Then, 
the ordinal patterns indicator vector
\[
\mathbf{Y}_t = \big[ \mathds{1}(\Pi(\mathbf{X}_t)=\pi_1),\dots,\mathds{1}(\Pi(\mathbf{X}_t)=\pi_{d!})\big]^\top .
\]
is a $2$-approximating functional of $(Z_t)_{t\in\mathbb{Z}}$ with constants
\[
b_k = d!\left(\varphi(\sqrt{2}\,a_k^{1/4}) + 8\sqrt{2}\,a_k^{1/4}\right),
\]
where
\[
\varphi(\varepsilon) = 4\,\mathbb{P}\!\big(\mathrm{ms}(\mathbf{X}_1)\leq 2\varepsilon\big),
\qquad 
\mathrm{ms}(\mathbf{x}) := \min_{1\leq i<j\leq d} |x_i - x_j|
\]
denotes the minimal spread of a vector $\mathbf{x} \in \mathbb{R}^d$.
\end{lemma}

\begin{proof}[Proof of Lemma \ref{lem:ordinal_vector_2approx}]
By Lemma~2.7 of ~\cite{Schnurr15092025}, the sliding-window process $(\mathbf{X}_t)$ is itself $2$-approximating with constants  $(a_k)_{k\geq 0}$. 
Each component of $\mathbf{Y}_t$ is of the form 
\[
Y_t^{(i)} = \mathds{1}\big(\Pi(\mathbf{X}_t)=\pi_i\big), \qquad i=1,\dots,d!,
\]
which is a bounded measurable map of $\mathbf{X}_t$. 
Proposition~4.2 of \cite{Schnurr15092025} shows that the indicator maps are $p$-continuous with 
\[
\varphi(\varepsilon)=2^p\,\mathbb{P}\!\big(\mathrm{ms}(\mathbf{X}_1)\leq 2\varepsilon\big).
\]
Applying Lemma~5.6 of \cite{Schnurr15092025} with $r=2$ yields that $(Y_t^{(i)})_{t\in \mathbb{Z}}$ 
is $2$-approximating with constants
\[
b_k = \varphi(\sqrt{2}a_k^{1/4}) + 8\sqrt{2}\,a_k^{1/4}.
\]
Finally, since
\[
\|\mathbf{Y}_t - \mathbb{E}[\mathbf{Y}_t\mid \mathcal{A}_{-m}^m]\|_2^2 
   \leq d!\max_{i=1,\dots,d!} |Y_t^{(i)} - \mathbb{E}[Y_t^{(i)}\mid \mathcal{A}_{-m}^m]|^2,
\]
it follows that $(\mathbf{Y}_t)$ is $2$-approximating with constants $d!\,b_k$. 
\end{proof}

\begin{remark}
\label{remark:sizes}
In general, the bound only guarantees $b_k \to 0$ as $k\to\infty$. 
However, if each marginal distribution of the differences $X_i-X_j$ ($0\le i<j\le d-1$) admits a Lipschitz continuous cumulative distribution function $F_{ij}$, with Lipschitz constant $L_{ij}$, then a sharper rate is available. Indeed,
\[
\{\mathrm{ms}(\mathbf{X}_1)\leq 2\varepsilon\}
   = \bigcup_{i<j}\{|X_i-X_j|\leq 2\varepsilon\},
\]
so that
\(
\mathbb{P}(\mathrm{ms}(\mathbf{X}_1)\leq 2\varepsilon) 
   \leq \sum_{i<j} \mathbb{P}(|X_i-X_j|\leq 2\varepsilon).
\)
Since $F_{ij}$ is Lipschitz with constant $L_{ij}$,
\[
\mathbb{P}(|X_i-X_j|\leq 2\varepsilon)
   = F_{ij}(2\varepsilon)-F_{ij}(-2\varepsilon)
   \leq 4L_{ij}\varepsilon.
\]
Hence,
\(
\varphi(\varepsilon) = 4\,\mathbb{P}(\mathrm{ms}(\mathbf{X}_1)\leq 2\varepsilon) 
   \leq 16\varepsilon \sum_{i<j} L_{ij}.
\)
Plugging into the expression for $b_k$ yields
\[
b_k \leq \Big(16\sqrt{2}\sum_{i<j} L_{ij}+8\sqrt{2}\Big)\,a_k^{1/4}.
\]
Therefore, under the Lipschitz assumption, $(b_k)_{k\geq 0}$ is of size $-\lambda/4$ whenever $(a_k)_{k\geq 0 }$ is of size $-\lambda$. 
\end{remark}

\begin{proposition} Let $\hat D_2(\mathcal{G})$ be the statistics defined in~\eqref{eq:estimator_asymmetry}. Then, 
\label{prop:decomposition}
\begin{enumerate}
    \item \label{prop_decomposition:item1} $\hat D_2(\mathcal{G})$ is a biased estimator of $D_2(\mathcal{G})$, and 
\[
\mathbb{E}[\hat D_2(\mathcal{G})]
= D_2(\mathcal{G}) + \sum_{i=1}^{d!} \text{Var}\!\big(\hat p_\mathcal{G}(\pi_i)\big).
\]
\item \label{prop_decomposition:item2}  $\hat{D}_2(\mathcal{G})$ is a U-statistics  up to a term of order $\mathcal{O}_P(n^{-1})$, i.e.,
\begin{equation}
\label{eq:definition_U_n^d}
\hat{D}_2(\mathcal{G}) 
= \frac{1}{n^2} \sum_{1 \leq k_1 \neq k_2 \leq n} h(\mathbf{X}_{k_1}, \mathbf{X}_{k_2}) 
+ \mathcal{O}_\mathbb{P}(n^{-1})
= U_n^d + \mathcal{O}_\mathbb{P}(n^{-1}),
\end{equation}
where $U_n^d$ denotes the $U$-statistic of order two associated with the symmetric kernel
\begin{equation}
\label{eq:kernel}
h(\mathbf{x}, \mathbf{y}) 
= \sum_{G \in \mathcal{G}} \frac{1}{|G|} 
   \mathds{1}\{\Pi(\mathbf{x}) \in G, \, \Pi(\mathbf{y}) \in G\} 
   - \mathds{1}\{\Pi(\mathbf{x}) = \Pi(\mathbf{y})\}.
\end{equation}
\end{enumerate}
\end{proposition}
\begin{proof}[Proof of Proposition \ref{prop:decomposition}]
Recall that
\[
\hat D_2(\mathcal{G}) 
= \sum_{i=1}^{d!} \big(\hat p_\mathcal{G}(\pi_i)\big)^2 - S_n^d,
\qquad 
D_2(\mathcal{G}) = \sum_{i=1}^{d!} \big(p_\mathcal{G}(\pi_i)\big)^2 - \sum_{i=1}^{d!} \big(p(\pi_i)\big)^2.
\]
We start by proving \ref{prop_decomposition:item1}
\cite{ caballero2019symbolic,Schnurr15092025} show that, under stationarity of $\mathbf{X}_t$, $S_n^d$ is unbiased estimator of $\sum_{i=1}^{d!} \big(p(\pi_i)\big)^2$, i.e.,
\[
\mathbb{E}[S_n^d] = \sum_{\pi \in \mathcal{S}_d} p(\pi)^2.
\]
Next, for each $\mathcal{G}$ and $\pi \in G$,
\[
\hat p_\mathcal{G}(\pi) 
= \frac{1}{|G|}\sum_{\pi' \in G} \hat p_n(\pi'),
\qquad
\hat p_n(\pi') = \frac{1}{n}\sum_{t=1}^n \mathds{1}\{\Pi(X_t) = \pi'\}.
\]
Since $\hat p_n(\pi')$ is unbiased for $p(\pi')$,
\(
\mathbb{E}[\hat p_\mathcal{G}(\pi)] = p_\mathcal{G}(\pi).
\)
Therefore,
\[
\mathbb{E}\big[(\hat p_\mathcal{G}(\pi))^2\big]
= \big(p_\mathcal{G}(\pi)\big)^2 + \text{Var}(\hat p_\mathcal{G}(\pi)).
\]
Summing over all $\pi \in \mathcal{S}_d$,
\[
\mathbb{E}\!\left[\sum_{i=1}^{d!} (\hat p_\mathcal{G}(\pi_i))^2\right]
= \sum_{i=1}^{d!} (p_\mathcal{G}(\pi_i))^2
+ \sum_{i=1}^{d!} \text{Var}(\hat p_\mathcal{G}(\pi_i)).
\]
Thus, by taking expectations in the definition of $\hat D_2(\mathcal{G})$,
\[
\mathbb{E}[\hat D_2(\mathcal{G})]
= \left(\sum_{i=1}^{d!} (p_\mathcal{G}(\pi_i))^2
- \sum_{i=1}^{d!} p(\pi_i)^2 \right)
+ \sum_{i=1}^{d!} \text{Var}(\hat p_\mathcal{G}(\pi_i)).
\]
The bracket equals $D_2(\mathcal{G})$, thus
\[
\mathbb{E}[\hat D_2(\mathcal{G})]
= D_2(\mathcal{G}) + \sum_{i=1}^{d!} \text{Var}(\hat p_\mathcal{G}(\pi_i)).
\]
In order to prove \ref{prop_decomposition:item2} a direct calculation shows:
\begin{align*}
 &\sum_{i=1}^{d!} (\hat{p}_\mathcal{G}(\pi_i))^2-S_n^d\\
 =&\frac{1}{n^2}\sum\limits_{G\in \mathcal{G}}
\frac{1}{|G|} 
\sum_{k_1=1}^{n}\sum_{k_2=1}^{n}\mathds{1}{\{\Pi(\mathbf{X}_{k_1}), \Pi(\mathbf{X}_{k_2})\in G\}}-\frac{1}{n(n-1)}\sum\limits_{1\leq k_1\neq k_2\leq n}\mathds{1}{\{\Pi(\mathbf{X}_{k_1})=\Pi(\mathbf{X}_{k_2})\}}\\
=&
 \frac{1}{n^2}\sum\limits_{G\in \mathcal{G}}
\frac{1}{|G|} 
\sum_{k=1}^{n}\mathds{1}{\{\Pi(\mathbf{X}_{k})\in G\}}
+ \frac{1}{n^2}\sum\limits_{G\in \mathcal{G}}
\frac{1}{|G|} 
\sum_{1\leq k_1\neq k_2\leq n}\mathds{1}{\{\Pi(\mathbf{X}_{k_1}), \Pi(\mathbf{X}_{k_2})\in G\}}\\
&-\frac{1}{n^2}\sum\limits_{1\leq k_1\neq k_2\leq n}\mathds{1}{\{\Pi(\mathbf{X}_{k_1})=\Pi(\mathbf{X}_{k_2})\}}\\
&+ \left(\frac{1}{n^2}-\frac{1}{n(n-1)}
\right)\sum\limits_{1\leq k_1\neq k_2\leq n}\mathds{1}{\{\Pi(\mathbf{X}_{k_1})=\Pi(\mathbf{X}_{k_2})\}}.
\end{align*}
We added and subtracted $\frac{1}{n^2}\sum\limits_{1\leq k_1\neq k_2\leq n}\mathds{1}{\{\Pi(\mathbf{X}_{k_1})=\Pi(\mathbf{X}_{k_2})\}}$ in the second equality. Note that first and last summand are both $\mathcal{O}_{\mathbb{P}}(n^{-1})$.
Then, the second and and third summand can be written together as
\begin{align*}
 U_n^d:=&   \frac{1}{n^2}
\sum_{1\leq k_1\neq k_2\leq n}\left(\sum\limits_{G\in \mathcal{G}}
\frac{1}{|G|} \mathds{1}{\{\Pi(\mathbf{X}_{k_1}), \Pi(\mathbf{X}_{k_2})\in G\}}-\mathds{1}{\{\Pi(\mathbf{X}_{k_1})=\Pi(\mathbf{X}_{k_2})\}} \right)
\\=&\frac{1}{n^2}
\sum_{1\leq k_1\neq k_2\leq n}h(\mathbf{X}_{k_1}, \mathbf{X}_{k_2})\;,
\end{align*}
where the kernel $h$ is defined by \eqref{eq:kernel}. 
\end{proof}
\begin{proposition}
\label{prop:kernel-properties}
Let $h:\mathbb{R}^d \times \mathbb{R}^d \to \mathbb{R}$ be the kernel defined in \eqref{eq:kernel}. Then: 
\begin{enumerate}
    \item \label{prop_kernel:item1}  Under the null hypothesis  $\mathcal{H}_0$, for two independent and identically distributed random vectors $\mathbf{X},\mathbf{Y} \in \mathbb{R}^d$, we have $\mathbb{E}[h(\mathbf{X},\mathbf{Y})]=0$. 
    \item \label{prop_kernel:item2}  Under the null hypothesis  $\mathcal{H}_0$, the kernel is degenerate, i.e.,\ $h_1(\mathbf{x}):=\mathbb{E}[h(\mathbf{x},\mathbf{Y})]=0$ for all $\mathbf{x}\in\mathbb{R}^d$. 
    \item \label{prop_kernel:item3} Under the alternative  $\mathcal{H}_1$, the kernel is non-degenerate, i.e.,\ there exists $\mathbf{x}\in\mathbb{R}^d$ such that $h_1(\mathbf{x})\neq 0$. 
\end{enumerate}
\end{proposition}
\begin{proof}[Proof of Proposition \ref{prop:kernel-properties}]
 We first compute
\begin{align*}
\mathbb{E}\!\left[\mathds{1}\{\Pi(\mathbf{X})=\Pi(\mathbf{Y})\}\right]
&= \sum_{\pi \in \mathcal{S}_d} \mathbb{E}\!\left[\mathds{1}\{\Pi(\mathbf{X})=\pi\}\,\mathds{1}\{\Pi(\mathbf{Y})=\pi\}\right] \\
&= \sum_{\pi \in \mathcal{S}_d} \mathbb{P}\big(\Pi(\mathbf{X})=\pi\big)\,\mathbb{P}\big(\Pi(\mathbf{Y})=\pi\big) \\
&= \sum_{\pi \in \mathcal{S}_d} \big(\mathbb{P}(\Pi(\mathbf{X})=\pi)\big)^2.
\end{align*}
Next, for the first term of the kernel,
\begin{align*}
\mathbb{E}\!\left[\sum_{G\in \mathcal{G}} \frac{1}{|G|}\,
\mathds{1}\{\Pi(\mathbf{X})\in G,\,\Pi(\mathbf{Y})\in G\}\right]
&= \sum_{G\in \mathcal{G}} \frac{1}{|G|}\,
\mathbb{E}\!\left[\mathds{1}\{\Pi(\mathbf{X})\in G,\,\Pi(\mathbf{Y})\in G\}\right] \\
&= \sum_{G\in \mathcal{G}} \frac{1}{|G|} \sum_{\pi_1\in G}\sum_{\pi_2\in G}
\mathbb{P}(\Pi(\mathbf{X})=\pi_1)\,\mathbb{P}(\Pi(\mathbf{Y})=\pi_2) \\
&= \sum_{G\in \mathcal{G}} \sum_{\pi\in G} \big(\mathbb{P}(\Pi(\mathbf{X})=\pi)\big)^2 \\
&= \sum_{\pi\in \mathcal{S}_d} \big(\mathbb{P}(\Pi(\mathbf{X})=\pi)\big)^2.
\end{align*}
Thus the two expectations coincide, so $\mathbb{E}[h(\mathbf{X},\mathbf{Y})]=0$ under  $\mathcal{H}_0$. 

\medskip
\ref{prop_kernel:item2} and \ref{prop_kernel:item3}  
Fix $\mathbf{x}\in \mathbb{R}^d$, and let $G_{\mathbf{x}}$ denote the unique set in $\mathcal{G}$ such that $\Pi(\mathbf{x})\in G_{\mathbf{x}}$. Then
\begin{align*}
h_1(\mathbf{x}) 
&:= \mathbb{E}[h(\mathbf{x},\mathbf{Y})] \\
&= \mathbb{E}\!\left[\frac{1}{|G_{\mathbf{x}}|}\,\mathds{1}\{\Pi(\mathbf{Y})\in G_{\mathbf{x}}\}\right]
- \mathbb{P}(\Pi(\mathbf{Y})=\Pi(\mathbf{x})) \\
&= \frac{1}{|G_{\mathbf{x}}|}\sum_{\pi\in G_{\mathbf{x}}} \mathbb{P}(\Pi(\mathbf{Y})=\pi)
- \mathbb{P}(\Pi(\mathbf{Y})=\Pi(\mathbf{x})).
\end{align*}
Under $\mathcal H_0$, the ordinal patterns distribution is constant on each
block $G\in\mathcal G$, that is,
\[
\mathbb{P}(\Pi(\mathbf Y)=\pi)
=
\mathbb{P}(\Pi(\mathbf Y)=\pi')
\qquad
\text{for all } \pi,\pi'\in G.
\]
In particular, for $G=G_{\mathbf x}$,
\[
\frac{1}{|G_{\mathbf{x}}|}
\sum_{\pi\in G_{\mathbf{x}}}
\mathbb{P}(\Pi(\mathbf Y)=\pi)
=
\mathbb{P}\!\left(\Pi(\mathbf Y)=\Pi(\mathbf x)\right).
\]
Consequently, $h_1(\mathbf x)=0$ for all $\mathbf x\in\mathbb{R}^d$, and the
kernel $h$ is degenerate.

Under  $\mathcal{H}_1$, we prove it by contraposition.  
Note that point 2. is an if and only. If $h_1(\mathbf{x}) = 0$ for all $\mathbf{x}$, then for each $\mathbf{x}$,
\[
p(\Pi(\mathbf{x})) 
=
\frac{1}{|G_{\mathbf{x}}|}
\sum_{\pi \in G_{\mathbf{x}}} p(\pi)
=:\bar p_{G}.
\]
Fix any $G\in\mathcal{G}$ and any $\pi,\pi' \in G$.  
There exist $\mathbf{x},\mathbf{x}'$ such that $\Pi(\mathbf{x})=\pi$ 
and $\Pi(\mathbf{x}')=\pi'$.  
Applying the above identity to $\mathbf{x}$ and $\mathbf{x}'$ yields
\[
p(\pi) = \bar p_G = p(\pi').
\]
Thus $p(\pi)$ is constant on $G$, and so the null hypothesis  $\mathcal{H}_0$ holds. We just proved that if $h_1(\mathbf{x}) = 0$ for all $\mathbf{x}$, then  $\mathcal{H}_0$ holds. 
\end{proof}

\begin{lemma}
\label{lemma:Hilber_operator_step_function}
Let $h$ be the kernel defined in  \eqref{eq:kernel}, and let $F$ denote cumulative distribution function of $\mathbf{X}_t$. Consider the integral operator
\[
\mathcal{A}:L^2(\mathbb{R}^d,dF) \longrightarrow L^2(\mathbb{R}^d,dF),
\qquad
(\mathcal{A}[g])(\mathbf{u}) := \int_{\mathbb{R}^d} h(\mathbf{u},\mathbf{v})\, g(\mathbf{v}) \, dF(\mathbf{v}).
\]
Then, $\mathcal{A}[g]$ is piece-wise constant over the sets $\{ \mathbf u \in \mathbb{R}^d : \Pi(\mathbf u)=\pi_i \},
\, i=1,\dots,d!$. More precisely, for every $g \in L^2(\mathbb{R}^d,dF)$ there exist 
$g_{\pi_1},\ldots,g_{\pi_{d!}} \in \mathbb{R}$ such that
\[
(\mathcal{A}[g])(\mathbf{u}) = \sum_{i=1}^{d!} g_{\pi_i} \, \mathds{1}\{ \Pi(\mathbf{u}) = \pi_i \}, \quad \text{for } \mathbf{u}\in \mathbb{R}^d\;.
\]
\end{lemma}
\begin{proof}
Fix \( \mathbf{u} \in \mathbb{R}^d \). Observe that the indicator function \( \mathds{1}\{ \Pi(\mathbf{u}), \Pi(\mathbf{v}) \in G \} \) is equal to 1 if and only if both \( \Pi(\mathbf{u}) \) and \( \Pi(\mathbf{v}) \) belong to the same group \( G \in \mathcal{G} \). In particular, for a fixed \( \mathbf{u} \), there is a unique element \( G_{\Pi(\mathbf{u})} \in \mathcal{G} \) such that \( \Pi(\mathbf{u}) \in G_{\Pi(\mathbf{u})} \), so:
\[
\sum\limits_{G \in \mathcal{G}} \frac{1}{|G|} \, \mathds{1}\{\Pi(\mathbf{u}), \Pi(\mathbf{v}) \in G\} = \frac{1}{|G_{\Pi(\mathbf{u})}|} \, \mathds{1}\{\Pi(\mathbf{v}) \in G_{\Pi(\mathbf{u})}\}.
\]
Therefore, 
\begin{align*}
    h(\mathbf{u}, \mathbf{v}) 
=& \sum_{G \in \mathcal{G}} \frac{1}{|G|} 
   \mathds{1}\{\Pi(\mathbf{u}) \in G, \, \Pi(\mathbf{v}) \in G\} 
   - \mathds{1}\{\Pi(\mathbf{u}) = \Pi(\mathbf{v})\}\\
=&\frac{1}{|G_{\Pi(\mathbf{u})}|} \, \mathds{1}\{\Pi(\mathbf{v}) \in G_{\Pi(\mathbf{u})}\} - \mathds{1}\{\Pi(\mathbf{v}) = \Pi(\mathbf{u})\}.
\end{align*}
The set \( \{ \mathbf{v} : \Pi(\mathbf{v}) \in G_{\Pi(\mathbf{u})} \} \) is a disjoint union of sets \( \{ \mathbf{v} : \Pi(\mathbf{v}) = \pi \} \) over \( \pi \in G_{\Pi(\mathbf{u})} \). On each of such sets, $h(\mathbf{u},\cdot)$ takes values
\[
h(\mathbf{u},\mathbf{v})=\begin{cases}
\frac{1}{|G_{\Pi(\mathbf{u})}|} - 1, & \text{if } \pi = \Pi(\mathbf{u}) \\
\frac{1}{|G_{\Pi(\mathbf{u})}|}, & \text{if } \pi \in G_{\Pi(\mathbf{u})} \setminus \{\Pi(\mathbf{u})\} \\
0, & \text{otherwise}.
\end{cases}\qquad \text{for all } \mathbf{v} \text{ such that } \Pi(\mathbf{v}) = \pi\;. 
\]
Consequently, \( \mathcal{A}[g](\mathbf{u}) \) is a weighted sum over the values of \( g \) evaluated at \( \mathbf{v} \) belonging to the same group as \( \mathbf{u} \). Since the sets \( \{ \mathbf{v} : \Pi(\mathbf{v}) = \pi \} \) partition \( \mathbb{R}^d \), and the kernel depends only on \( \Pi(\mathbf{v}) \), we can express the operator as:
\begin{align}
\mathcal{A}[g](\mathbf{u}) = \sum_{\pi \in G_{\Pi(\mathbf{u})}} \left( \frac{1}{|G_{\Pi(\mathbf{u})}|} - \delta_{\pi, \Pi(\mathbf{u})} \right) \int_{\{\mathbf{v}\,:\, \Pi(\mathbf{v}) = \pi \}} g(\mathbf{v}) \, dF(\mathbf{v}),
\label{eq:operator_explicit}
\end{align}
where \( \delta_{\pi, \Pi(\mathbf{u})} \) is the Kronecker delta. This means that for each permuatation $\pi \in \mathcal{S}_d$, the value $\mathcal{A}[g](\mathbf{u})$ is constant over the set $\{ \Pi(\mathbf{u})=\pi \}$, as the right hand side of \eqref{eq:operator_explicit} only depends on $G_{\Pi(\mathbf{u})}$ -- every vector in $\{ \Pi(\mathbf{u})=\pi \}$ is associated with the same group. Let $G_\pi$ be the unique element in $\mathcal{G}$ such that $\pi \in G_\pi$.
\[
g_{\pi} := \sum_{\pi' \in G_{\pi}} \left( \frac{1}{|G_{\mathbf{\pi}}|} - \delta_{\pi',\pi} \right)  \int_{\{ \mathbf{v}\,:\,\Pi(\mathbf{v}) = \pi \}} g(\mathbf{v}) \, dF(\mathbf{v})\;,
\]
 \( \mathcal{A}[g](\mathbf{u}) \) reduces to:
\[
\mathcal{A}[g](\mathbf{u}) = \sum_{i=1}^{d!} g_{\pi_i} \, \mathds{1}\{\Pi(\mathbf{u}) = \pi_i\},
\]
which concludes the proof.
\end{proof}
The operator $\mathcal{A}$ is an integral operator with square integrable kernel, and therefore  it is a Hilbert–Schmidt operator. Every Hilbert–Schmidt operator on a Hilbert space is compact. Moreover the kernel $h$ is symmetric hence  $\mathcal{A}$ is self-adjoint. 
By the spectral theorem, $\mathcal{A}$ admits a countable orthonormal basis of eigenfunctions $\{g^{(i)}\}_{i \geq 1}$ in $L^2(\mathbb{R}^d,dF)$, with corresponding real eigenvalues $\{\lambda_i\}_{i \geq 1}$. Concretely, $(\lambda,g)$ is an eigenvalue–eigenfunction pair if
\[
\mathcal{A}[g] = \lambda g 
\quad \Longleftrightarrow \quad 
\int_{\mathbb{R}^d} h(\mathbf{u},\mathbf{v})\, g(\mathbf{v}) \, dF(\mathbf{v}) = \lambda g(\mathbf{u})
\quad \text{for $F$-almost every $\mathbf{u} \in \mathbb{R}^d$.}
\]
The eigenfunctions $\{g^{(i)}\}_{i\geq 1}$ form an orthonormal system, i.e.,
\[
\int_{\mathbb{R}^d} g^{(i)}(\mathbf{z})\, g^{(j)}(\mathbf{z}) \, dF(\mathbf{u}) = \delta_{ij},
\]
where $\delta_{ij}$ is the Kronecker delta. These spectral quantities determine the limiting distribution of the statistic under  $\mathcal{H}_0$.  The next proposition characterizes the pairs $(\lambda, g)$ for $\mathcal{A}$ under the null hypothesis.  We remind the reader of the following notation, needed for Proposition \ref{proposition:eigenvalues}.
For any vector \( \mathbf{u} \in \mathbb{R}^d \), we denote by \( G_{\Pi(\mathbf{u})} \) the unique element of \( \mathcal{G} \) such that \( \Pi(\mathbf{u}) \in G_{\Pi(\mathbf{u})} \); this is well-defined by construction. $\mathcal{G}=\{G_1, \ldots, G_m\}$ with $|G_1|=d_1, \ldots, |G_m|=d_m$, i.e., $\sum d_i=d!$, and $G_i=\{\pi_{i,1}, \ldots, \pi_{i,{d_i}}\}$. For $\pi_{i,1},\pi_{i,2} \in G_i$, $ p_\mathcal{G}(\pi_{i,1})=p_\mathcal{G}(\pi_{i,2})$, thus we define $p_i$ as $p_i:=p_\mathcal{G}(\pi) $ for any $\pi \in G_i\;.$ Under the null hypothesis $p_i=p(\pi_i)\;.$ Lastly, $\mathbf{p}=[p_1, \ldots, p_m]^\top\;.$
\begin{proposition}
\label{proposition:eigenvalues}
Under the null hypothesis, the nonzero elements of the spectrum of the operator $\mathcal A$
consist of the eigenvalues $-p_i$, $i=1,\ldots,m$.
For each $i$, the eigenspace associated with the eigenvalue $-p_i$ is
$V_{-p_i}$, given by
\begin{align*}
V_{-p_i}
&:= \operatorname{span}(B_{-p_i}), \\[0.5em]
B_{-p_i}
&:= \Bigg\{
\sum_{r=1}^{m}\sum_{j=1}^{d_r}
g_{\pi_{r,j}}\,
\mathds{1}\!\left\{\Pi(\cdot)=\pi_{r,j}\right\}
\;\Bigg|\;
\mathbf g_r := (g_{\pi_{r,1}},\ldots,g_{\pi_{r,d_r}})\in\widetilde B_r,
\\[-0.2em]
&\hspace{7.8em}
\mathbf g_r = \mathbf 0 \text{ for all } r\neq i
\Bigg\}\;.
\end{align*}
$\widetilde{B}_{-p_i} \subset \mathbb{R}^{d_i}$ is the set of $d_i-1$ elements
$$
\widetilde{B}_{-p_i} :=  \left\{ \frac{1}{\sqrt{p_i\, j(j+1)}} \left( \sum_{k=1}^j \mathbf{e}_k - j\, \mathbf{e}_{j+1} \right) \;,\; j = 1, \ldots, d_i - 1 \right\},
$$
and $\{\mathbf{e}_k\}_{k=1}^{d_i}$ denotes the standard basis of $\mathbb{R}^{d_i}$. Moreover,  $V_{-p_i}\perp V_{-p_j}$  for $i\neq j$ and $B_{-p_i}$ forms a collection of orthonormal vectors with respect to the scalar product \(
\langle g_1, g_2 \rangle_{L^2(\mathbb{R}^d, dF)} := \int g_1(\mathbf{x}) g_2(\mathbf{x}) \, dF(\mathbf{x})
\). Lastly, $V_{-p_i}$ is isomorphic to $\text{Span}\{\tilde{B}_{-p_i}\}$ and $\dim (V_{-p_i})=d_i-1\;.$
\end{proposition}
\begin{remark}
    Note that each $\mathbf{g}_i$ depends on $p_i$, thus $\mathbf{g}_i=\mathbf{g}_i(p_i)\;.$
\end{remark}
\begin{proof}[Proof Proposition \ref{proposition:eigenvalues}]
Equation \eqref{eq:operator_explicit} (proof of Lemma \ref{lemma:Hilber_operator_step_function}) shows that 
\begin{align*}
    \mathcal{A}[g](\mathbf{u}) = \sum_{\pi' \in G_{\Pi(\mathbf{u})}} \left( \frac{1}{|G_{\Pi(\mathbf{u})}|} - \delta_{\pi', \Pi(\mathbf{u})} \right) \int_{\{\mathbf{v}\,:\, \Pi(\mathbf{v}) = \pi' \}} g(\mathbf{v}) \, dF(\mathbf{v})\;.
\end{align*}
 By substituting $g(\mathbf{u})=\sum_{i=1}^{d!} g_{\pi_i} \, \mathds{1}\{\Pi(\mathbf{u}) = \pi_i\}$ into the integral, and noting that $g(\mathbf{u})=g_{\Pi(\mathbf{u})}$,  the eigenvalue  equation $ \mathcal{A}[g](\mathbf{u})=\lambda  g(\mathbf{u})$ can be written as
 \begin{align*}
     \frac{1}{|G_{\Pi(\mathbf{u})}|}\sum\limits_{\pi'\in G_{\Pi(\mathbf{u})}}g_{\pi'}p(\pi')-g_{\Pi(\mathbf{u})}p\left(\Pi(\mathbf{u})\right)=\lambda g_{\Pi(\mathbf{u})}\;.
 \end{align*}
Under the hypothesis of symmetry with respect to $\mathcal{G}$
 \begin{align*}
     \frac{1}{|G_{\Pi(\mathbf{u})}|}p\left(\Pi(\mathbf{u})\right)\sum\limits_{\pi\in G_{\Pi(\mathbf{u})}}g_{\pi}-g_{\Pi(\mathbf{u})}p\left(\Pi(\mathbf{u})\right)=\lambda g_{\Pi(\mathbf{u})}\;,\qquad \text{for }\mathbf{u} \in \mathbb{R}^k\;.
 \end{align*}
 As $\mathbf{u} $ varies in $\mathbb{R}^k$, we obtain in total $d!$ different equations.
By defining 
\begin{align*}
\mathbf{A}_i:=\begin{bmatrix}
p_i(\frac{1}{d_i} -1)& \frac{p_i}{d_i}& \cdots & \frac{p_{i}}{d_i} \\
\frac{p_{i}}{d_i} & p_{i}(\frac{1}{d_i} -1)& \cdots & \frac{p_{i}}{d_i} \\
\vdots & \vdots & \ddots & \vdots \\
\frac{p_{i}}{d_i} & \frac{p_{i}}{d_i}& \cdots &  p_{i}(\frac{1}{d_i} -1)
\end{bmatrix}\;,\quad \mathbf{g}_i:=\begin{bmatrix}
g_{\pi_{i,1}}\\
g_{\pi_{i,2}}\\
\vdots \\
g_{\pi_{i,{d_i}}}
\end{bmatrix}
\end{align*}
we obtain the following system of equations:
\begin{align}
\label{eq:linear_system}
\begin{bmatrix}
\mathbf{A}_1& 0 & \cdots & 0\\
0 & \mathbf{A}_2 & \cdots & 0 \\
\vdots & \vdots & \ddots & \vdots \\
0 & 0 & \cdots &  \mathbf{A}_m
\end{bmatrix}
\begin{bmatrix}
\mathbf{g}_{1}\\
\mathbf{g}_{2}\\
\vdots \\
\mathbf{g}_{m}
\end{bmatrix}
=
\lambda 
\begin{bmatrix}
\mathbf{g}_{1}\\
\mathbf{g}_{2}\\
\vdots \\
\mathbf{g}_{m}
\end{bmatrix}
\end{align}
By defining $\mathbf{A}=\text{diag}(\mathbf{A}_1, \ldots, \mathbf{A}_m)$ and $\mathbf{g}=\text{vec}(\mathbf{g}_1,\ldots, \mathbf{g}_m)$, \eqref{eq:linear_system} can be written in the compact form $\mathbf{A} \mathbf{g}=\lambda \mathbf{g}\;.$
Further, let \( a_i = \frac{p_i}{d_i} \) and \( \mathbf{1} = (1, \ldots, 1)^\top \). With this notation, each matrix $\mathbf{A}_i$ can be rewritten as  
\[
\mathbf{A}_i = a_i \, \mathbf{1}\mathbf{1}^\top - p_i \mathbf{I}_{d_i}.
\]
The eigenvalues of a matrix of the form \( \mathbf{M} + \beta \mathbf{I} \) are simply the eigenvalues of \( \mathbf{M} \) shifted by \( \beta \). In fact, if $\mathbf{M}\in \mathbb{R}^{n \times n} $ has an eigenpair \( (\lambda, \mathbf{v}) \), with \( \mathbf{v} \in \mathbb{R}^n \), then, for any scalar \( \beta \neq 0 \), \( (\lambda + \beta, \mathbf{v}) \) is an eigenpair of \( \mathbf{M} + \beta \mathbf{I} \). This is due to the fact that
\[
(\mathbf{M} + \beta \mathbf{I})\mathbf{v} = \mathbf{M}\mathbf{v} + \beta \mathbf{v}= \lambda \mathbf{v} + \beta \mathbf{v}= (\lambda + \beta) \mathbf{v}.
\]
Returning to the matrix \( \mathbf{A}_i \),  \( \mathbf{1}\mathbf{1}^\top \) is a rank-one matrix. The eigenvalues of $ \mathbf{1}\mathbf{1}^\top$  are easily characterized: it has a single non-zero eigenvalue equal to \( d_i \), corresponding to the eigenvector \( \mathbf{1} \), and the eigenvalue \( 0 \) with multiplicity \( d_i - 1 \). Therefore, the eigenvalues of \( \mathbf{A}_i \) can be computed by shifting the eigenvalues of \( a_i \, \mathbf{1}\mathbf{1}^\top \) by \( -p_i \). Concretely, \( \mathbf{A}_i \) has eigenvalues \( a_i d_i - p_i = 0 \) (with multiplicity 1) and \( -p_i \) (with multiplicity \( d_i - 1 \)).  as it is also the eigenspace of $\mathbf{1}\mathbf{1}^\top$ for eigenvalue $0$. 

For $j=1, \ldots, d_i$, define the vectors $\mathbf{v}_{i,j} \in \mathbb{R}^{d!}$ as 

\[
\mathbf{v}_{i,j}
=
\left[\,0,\ldots,0,
-1,
0,\ldots,0,
1,
0,\ldots,0\,\right]^\top ,
\]
that is, $\mathbf v_{i,j}$ has exactly two nonzero entries, both located in the $i$-th block: a $-1$ in its first coordinate and a $+1$ in its $(j+1)$-th coordinate. Therefore, the eigenspace associated with $p_i$ is 
$$ \tilde V_{-p_i}= \text{Span}\left\{ \mathbf{v}_{i,1}, \ldots,  \mathbf{v}_{i,d_i} \right\}\subset \mathbb{R}^{d!}\;, \quad \text{dim}(V_{-p_i})=d_i-1\;.$$
The eigenspace $\tilde V_{-p_i}$ is the orthogonal vector space to the  Span$\{\mathbf{1}\}.$
Let us define the subspace
\[
\tilde V := \text{Span}\left( \{\mathbf{g}_{i,j}\}_{i=1,\ldots,m;\; j=1,\ldots,d_i-1} \right) \subset \mathbb{R}^{d!},
\]
where the vectors \(\mathbf{g}_{i,j}\) are the eigenvectors of \(\mathbf{A}\).
Let $(k_s)_{s=1,\ldots,d!}$ be the sequence of indexes that sorts the elements in the partition. For instance, $[g_{\pi_{k_1}}, \ldots, g_{\pi_{k_{d_i}}}]=[g_{\pi_{1,1}}, \ldots,g_{\pi_{1,d_{i_1}}}]=\mathbf{g}_{1}\;,[ g_{\pi_{k_{d_1+1}}}, \ldots, g_{\pi_{k_{d_1+d_2}}}]=[g_{\pi_{2,1}}, \ldots, g_{\pi_{2,d_2}}]=\mathbf{g}_2$ etc.

As discussed above, the corresponding eigenspace of eigenfunctions for the operator \(\mathcal{A}\) consists of functions of the form
\begin{align}
\label{eq:V_function}
V_{\text{func}} := \left\{ \sum_{i=1}^{d!} g_{\pi_{k_i}} \mathds{1}\big( \Pi(\cdot) = \pi_{k_i} \big) \;\middle|\; 
\begin{pmatrix}
g_{\pi_{k_1}} \\
\vdots \\
g_{\pi_{k_{d!}}}
\end{pmatrix} =
\begin{pmatrix}
\mathbf{g}_{1}\\
\vdots \\
\mathbf{g}_{m}
\end{pmatrix}\in \tilde V \right\}.
\end{align}
We define the linear map
\begin{align}
\label{eq:isomorphism}
    f: \tilde V &\longrightarrow V_{\text{func}} \subset L^2(\mathbb{R}^d, dF) \\
    \mathbf{g} &\longmapsto \sum_{i=1}^{d!} g_{\pi_{k_i}} \mathds{1}\big( \Pi(\cdot) = \pi_{k_i} \big)\;,
\end{align}
and $V_{-p_i}:=f(\tilde V_{-p_i})$. The map $f$  is a vector space isomorphism, but it is not an isometry. To see this, recall that \(V\) is equipped with the standard Euclidean inner product, while \(V_{\text{func}}\) inherits the \(L^2(\mathbb{R}^d, dF)\) inner product:
\[
\langle g^{(1)}, g^{(2)} \rangle_{L^2(\mathbb{R}^d, dF)} := \int g^{(1)}(\mathbf{x}) g^{(2)}(\mathbf{x}) \, dF(\mathbf{x}).
\]
Let \(\mathbf{g}^{(1)}, \mathbf{g}^{(2)} \in V\) with \(\mathbf{g}^{(1)} \perp \mathbf{g}^{(2)}\) in the Euclidean sense. Define the corresponding functions \(g^{(1)} := f(\mathbf{g}^{(1)})\) and \(g^{(2)} := f(\mathbf{g}^{(2)})\). Then,
\[
\begin{aligned}
\langle g^{(1)}, g^{(2)} \rangle_{L^2(\mathbb{R}^d, dF)} &= \left\langle f(\mathbf{g}^{(1)}), f(\mathbf{g}^{(2)}) \right\rangle_{L^2(\mathbb{R}^d, dF)} \\
&= \int \sum_{i=1}^{d!} \sum_{j=1}^{d!} g^{(1)}_{\pi_{k_i}} g^{(2)}_{\pi_{k_j}} \mathds{1}\big(\Pi(\mathbf{u}) = \pi_{k_i}\big) \mathds{1}\big(\Pi(\mathbf{u}) = \pi_{k_j}\big) \, dF(\mathbf{u}) \\
&= \sum_{j=1}^{d!} g^{(1)}_{\pi_{k_j}} g^{(2)}_{\pi_{k_j}} \, \mathbb{P} \big( \Pi(\mathbf{X}_1) = \pi_{k_j} \big) \quad \text{(since the indicator functions are disjoint)} \\
&= (\mathbf{g}^{(1)})^\top \left[ \begin{array}{ccc}
      p(\pi_{k_1}) &        & \\
               & \ddots & \\
               &        & p(\pi_{k_{d!}})
\end{array} \right] \mathbf{g}^{(2)}.
\end{aligned}
\]
The eigenvalues and eigenvectors of \(\mathbf{A}\) are derived under the the null hypothesis \(\mathcal{H}_0\), In particular, the inner product becomes
\[
\langle g_1, g_2 \rangle_{L^2(\mathbb{R}^d, dF)} = (\mathbf{g}^{(1)})^\top 
\setlength{\fboxsep}{1pt}
\left[
\begin{array}{@{}c@{}c@{}c@{}}
\underbrace{\fbox{$
  \begin{array}{ccc}
  p_1 &        &     \\
      & \ddots &     \\
      &        & p_1
  \end{array}
$}}_{\!\!\!d_1\times d_1} 
& & \\[2pt]

& \ddots & \\[2pt]

& & \underbrace{\fbox{$
  \begin{array}{ccc}
  p_m &        &     \\
      & \ddots &     \\
      &        & p_m
  \end{array}
$}}_{\!\!\! d_m\times d_m}
\end{array}
\right]
\mathbf{g}^{(2)}.
\]
Thus, a normalization of $\mathbf{g}_{i,j}$ by $\sqrt{p_i}$ yields orthonormality, concluding the proof.  
\end{proof}

\begin{lemma}
\label{lemma:mean_zero_eigenfunction}
Let \( h \in L^2(\mathbb{R}^{2d},d F \times dF) \) be a symmetric kernel, i.e., \( h(\mathbf{u}, \mathbf{v}) = h(\mathbf{v}, \mathbf{u}) \), that is degenerate:
\[
\int h(\mathbf{u}, \mathbf{v}) \, dF(\mathbf{u}) = 0 \quad \text{almost  surely for all } \mathbf{v} \in \mathbb{R}^d.
\]
Let \( \mathcal{A}:L^2(\mathbb{R}^d,dF) \longrightarrow L^2(\mathbb{R}^d,dF), \) be the Hilbert--Schmidt integral operator defined by
\[
\mathcal{A} [g](\mathbf{u}) := \int h(\mathbf{u}, \mathbf{v}) g(\mathbf{v}) \, dF(\mathbf{v}),
\]
and let \( g_i \in L^2(\mathbb{R}^d,dF) \) be an eigenfunction of \( \mathcal{A} \) corresponding to a nonzero eigenvalue \( \lambda_i \neq 0 \), i.e.,
\[
\mathcal{A}[g^{(i)}] = \lambda_i g^{(i)}.
\]
Then, for $\mathbf{Z}\sim F$ and \( g^{(i)}(\mathbf{Z)} \) 
\[
\mathbb{E}[g^{(i)}(\mathbf{Z})] = \int g^{(i)}(\mathbf{u}) \, dF(\mathbf{u}) = 0.
\]
\end{lemma}

\begin{proof}
Since \( h \) is symmetric and square-integrable, the operator \( \mathcal{A} \) is self-adjoint and Hilbert--Schmidt on \( L^2(\mathbb{R}^d,dF) \). In particular, its eigenfunctions form an orthonormal basis of \( L^2(\mathbb{R}^d,dF) \), and each eigenfunction satisfies $
\mathcal{A}[g^{(i)}] = \lambda_i g^{(i)}.$
We integrate both sides of the eigenvalue equation:
\begin{equation}
\label{eq:eigenvalued_integrated}
    \int \mathcal{A}[g^{(i)}](\mathbf{u}) \, dF(\mathbf{u}) = \lambda_i \int g^{(i)}(\mathbf{u}) \, dF(\mathbf{u})=\lambda_i \mathbb{E}[g^{(i)}(\mathbf{Z})].
\end{equation}
Using the definition of \( \mathcal{A} \) and applying Fubini's theorem (justified since \( h \in L^2(\mathbb{R}^{2d},dF \times dF) \)), we interchange the order of integration:
\[
\int \mathcal{A}[g^{(i)}](\mathbf{u}) \, dF(\mathbf{u}) 
= \int \left( \int h(\mathbf{u}, \mathbf{v}) g^{(i)}(\mathbf{v}) \, dF(\mathbf{v}) \right) dF(\mathbf{u})
= \int \left( \int h(\mathbf{u}, \mathbf{v}) \, dF(\mathbf{u}) \right) g^{(i)}(\mathbf{v}) \, dF(\mathbf{v}).
\]
By the degeneracy assumption, we have
\[
\int h(\mathbf{u}, \mathbf{v}) \, dF(\mathbf{u}) = 0 \quad \text{for almost every } \mathbf{v}\in \mathbb{R}^d,
\]
so the entire expression evaluates to zero. 
Therefore, from the eigenvalue equation
\[
\lambda_i \int g^{(i)}(\mathbf{u}) \, dF(\mathbf{u}) = 0.
\]
Since \( \lambda_i \neq 0 \) by assumption, it follows from \eqref{eq:eigenvalued_integrated} that $\mathbb{E}[g^{(i)}(\mathbf{Z})]=0$. 
\end{proof}

\section{Proofs}
\label{appendix:proofs}
\begin{proof}[Proof of Theorem \ref{theorem:asymptotic_distribution}]
By Proposition \ref{prop:decomposition}, $\hat{D}_2(\mathcal{G}) = U_n^d + \mathcal{O}(1/n) $. In our case, the kernel $h$ in \eqref{eq:kernel} associated with $U_n^d$ is symmetric, bounded, and square-integrable. By Proposition~\ref{prop:kernel-properties} (i), $h$ is centered, i.e.,\ $\mathbb{E}[h(\mathbf{X},\mathbf{Y})]=0$ for i.i.d.\ $\mathbf{X},\mathbf{Y}$. Thus $h$ qualifies as a Hilbert-Schmidt kernel on $L^2(\mathbb{R}^d,dF)$, where $dF$ denotes the measure associated (or induced) by the marginal distribution $F$ of $\mathbf{X}_t$. Thus the associated integral operator 
\[
\mathcal{A}:L^2(\mathbb{R}^d,dF) \longrightarrow L^2(\mathbb{R}^d,dF),
\qquad
(\mathcal{A}[g])(\mathbf{u}) := \int_{\mathbb{R}^d} h(\mathbf{u},\mathbf{v})\, g(\mathbf{v}) \, dF(\mathbf{v})\;,
\]
is compact and self-adjoint. By the spectral theorem, $\mathcal{A}$ admits a countable orthonormal basis of eigenfunctions $\{g^{(i)}\}_{i \geq 1}$ in $L^2(\mathbb{R}^d,dF)$, with corresponding real eigenvalues $\{\lambda_i\}_{i \geq 1}$. Concretely, $(\lambda,g)$ is an eigenvalue–eigenfunction pair if
\[
\mathcal{A}[g] = \lambda g 
\quad \Longleftrightarrow \quad 
\int_{\mathbb{R}^d} h(\mathbf{u},\mathbf{v})\, g(\mathbf{v}) \, dF(\mathbf{v}) = \lambda g(\mathbf{u})
\quad \text{for $F$-almost every $\mathbf{u} \in \mathbb{R}^d$.}
\]
The eigenfunctions $\{g^{(i)}\}_{i\geq 1}$ form an orthonormal system, i.e.,
\[
\int_{\mathbb{R}^d} g^{(i)}(\mathbf{z})\, g^{(j)}(\mathbf{z}) \, dF(\mathbf{u}) = \delta_{ij},
\]
where $\delta_{ij}$ is the Kronecker delta. Thus, all hypothesis of Theorem 2 of \cite{carlstein1988degenerate} are satisfies. 
     Lastly, from the proof of Proposition \ref{prop:decomposition}, the $c_n=\mathcal{O}_\mathbb{P}(1/n)$ terms are 
     \[
c_n:= \frac{1}{n^2}\sum_{G\in \mathcal{G}}\frac{1}{|G|}\sum_{k=1}^{n}\mathds{1}\{\Pi(\mathbf{X}_{k})\in G\} -\frac{1}{n^2(n-1)}\sum_{1\leq k_1\neq k_2\leq n}\mathds{1}\{\Pi(\mathbf{X}_{k_1})=\Pi(\mathbf{X}_{k_2})\}.
\]
The convergence of $nc_n$ follows from Theorem U of \cite{AaronsonEtAl1996}. 

\end{proof}

\begin{proof}[Proof of Lemma \ref{lem:variance_dependence}]
      By the spectral theorem, any real symmetric matrix $\boldsymbol{\Sigma} \in \mathbb{R}^{n \times n}$ admits a decomposition $\boldsymbol{\Sigma} = P^\top D P$, where $P$ is orthogonal and $D = \mathrm{diag}(\lambda_1, \ldots, \lambda_n)$ with $\lambda_j \geq 0$. Define $\sqrt{D} := \mathrm{diag}(\sqrt{\lambda_1}, \ldots, \sqrt{\lambda_n})$ and set $\sqrt{\boldsymbol{\Sigma}} := P^\top \sqrt{D} P$. Then $\sqrt{\boldsymbol{\Sigma}}^2 = \boldsymbol{\Sigma}$.
    
We apply the previous argument on $\mathbf{W}$, which allows us to express the distribution of $S_t(\mathbf{p})$ in terms of a weighted sum of independent chi-squared variables. Since $\mathbf{W} \sim \mathcal{N}(\mathbf{0}, \boldsymbol{\Sigma}(\mathbf{p}))$, there exists a standard Gaussian vector $\mathbf{R} \sim \mathcal{N}(\mathbf{0}, \mathbf{I})$ such that $\mathbf{W} \overset{\mathcal{D}}{=} \boldsymbol{\Sigma}^{1/2}(\mathbf{p}) \mathbf{R}$, where $\boldsymbol{\Sigma}^{1/2}(\mathbf{p})$ is any symmetric square root of $\boldsymbol{\Sigma}(\mathbf{p})$. 
The square root is not unique, but this does not matter here, since we are only interested in the distribution of the quadratic form.
Substituting $\mathbf{W} \overset{\mathcal{D}}{=} \boldsymbol{\Sigma}^{1/2}(\mathbf{p}) \mathbf{R}$ into the expression for $S_t(\mathbf{p})$, we obtain
$$
S_t(\mathbf{p}) \overset{\mathcal{D}}{=} \mathbf{R}^\top \boldsymbol{\Sigma}^{1/2}(\mathbf{p}) \boldsymbol{\Lambda}(\mathbf{p}) \boldsymbol{\Sigma}^{1/2}(\mathbf{p}) \mathbf{R} - \mathrm{tr}(\boldsymbol{\Lambda}(\mathbf{p})).
$$
Let $\mathbf{B}(\mathbf{p}) := \boldsymbol{\Sigma}^{1/2}(\mathbf{p}) \boldsymbol{\Lambda}(\mathbf{p}) \boldsymbol{\Sigma}^{1/2}(\mathbf{p})$. Then $\mathbf{B}(\mathbf{p}) \in \mathbb{R}^{t \times t}$ is symmetric and real-valued, so by the spectral theorem, there exists an orthogonal matrix $U(\mathbf{p})$ such that
$$
\mathbf{B}(\mathbf{p}) = U(\mathbf{p})^\top \mathbf{M}(\mathbf{p}) U(\mathbf{p}),
$$
where $\mathbf{M}(\mathbf{p}) = \mathrm{diag}(\mu_1(\mathbf{p}), \ldots, \mu_t(\mathbf{p}))$ contains the eigenvalues of $\mathbf{B}(\mathbf{p})$. Therefore,
$$
S_t(\mathbf{p}) \overset{\mathcal{D}}{=} \mathbf{R}^\top U(\mathbf{p})^\top \mathbf{M}(\mathbf{p}) U(\mathbf{p}) \mathbf{R} - \mathrm{tr}(\boldsymbol{\Lambda}(\mathbf{p})).
$$
Since orthogonal transformations preserve the standard normal distribution, the random vector $\mathbf{Y}=[Y_1, \ldots, Y_t]^\top$ defined as $\mathbf{Y} := U(\mathbf{p}) \mathbf{R}$ satisfies $\mathbf{Y} \sim \mathcal{N}(\mathbf{0}, \mathbf{I})$. 
Thus, we conclude
$$
S_t(\mathbf{p}) \overset{\mathcal{D}}{=} \sum_{j=1}^t \mu_j(\mathbf{p}) Y_j^2 - \mathrm{tr}(\boldsymbol{\Lambda}(\mathbf{p})).
$$
The random variable $S_t(\mathbf{p}) = \sum_{j=1}^t \mu_j(\mathbf{p}) Y_j^2 - \sum_{j=1}^t \lambda_j(\mathbf{p})$ has a generalized chi-squared distribution. The eigenvalues $\mu_j(\mathbf{p})$ govern the shape, while the second summand shifts the distribution. This already suggests that the distribution of $S_t(\mathbf{p})$ depends on $\mathbf{p}$. We now formalize this intuition via the characteristic function.

Let $\varphi_{S_t}(u,\mathbf{p}) := \mathbb{E}[e^{i u S_t(\mathbf{p})}]$ denote the characteristic function of $S_t(\mathbf{p})$. Since the variables $Y_j$ are independent by construction, and the characteristic function of a chi-squared distributed variables is known,  a direct computation gives:
\begin{equation}
\label{eq:characteristic_function}
    \varphi_{S_t}(u,\mathbf{p}) = \exp\left(-i u \sum_{j=1}^t \lambda_j(\mathbf{p})\right)
\prod_{j=1}^t \left(1 - 2 i u \mu_j(\mathbf{p})\right)^{-1/2}.
\end{equation}
To prove that the distribution of $S_t(\mathbf{p})$ depends on $\mathbf{p}$, it is equivalent to show that its characteristic function $\varphi_{S_t}(u,\mathbf{p})$ is not constant in the second argument for some fixed $u \in \mathbb{R}$.
Thus, without loss of generality, we fix $u \in \mathbb{R} \setminus \{0\}$ for the remainder of the argument.
We distinguish two cases:\\
\textbf{Case 1}: $\mu_j(\mathbf{p}) \equiv \text{const}$ for all $j = 1, \ldots, t$.
In this case, the dependence of $\varphi_{S_t}(u,\mathbf{p})$ on $\mathbf{p}$ is already guaranteed by the term $\sum_{j=1}^t \lambda_j(\mathbf{p})$, since the eigenvalues $\lambda_j(\mathbf{p})$ depend on $\mathbf{p}$. In fact, by Proposition \ref{proposition:eigenvalues}, the multiplicity of each $p_i$ is $|G_i|-1$, thus
\[
\sum_{j=1}^t \lambda_j(\mathbf{p})
= -\sum_{i=1}^m (|G_i|-1)\,p_i.
\]
Recalling that for the grouped probabilities holds $
\sum_{i=1}^m |G_i|\,p_i = 1, $
we may rewrite
\[
\sum_{j=1}^t \lambda_j(\mathbf{p})
= -\sum_{i=1}^m |G_i|\,p_i + \sum_{i=1}^m p_i
= -1 + \sum_{i=1}^m p_i.
\]

Hence, except in the case where $|G_1|=\cdots=|G_m|$, the quantity
$\sum_{j=1}^t \lambda_j(\mathbf{p})$ depends nontrivially on $\mathbf{p}$.
Consequently, $\varphi_{S_t}(u,\mathbf{p})$ is not constant in $\mathbf{p}$. \\

\textbf{Case 2}: There exists $j \in \{1, \ldots, t\}$ such that $\mu_j(\mathbf{p})$ is not constant.
We proceed by contradiction. Suppose that the distribution of $S_t(\mathbf{p})$ does not depend on $\mathbf{p}$, i.e $\mathbf{p}\mapsto  \varphi_{S_t}(u,\mathbf{p})\equiv\text{const}$. Then, the function 
$$
f_t(u,\mathbf{p}) := \left| \varphi_{S_t}(u,\mathbf{p}) \right|^{-1/4} = \prod_{j=1}^t \left(1 + 4 u^2 \mu_j(\mathbf{p})^2 \right).
$$
 is constant in the second argument as well. In other words,
$$
\mathbf{p} \mapsto f_t(u,\mathbf{p}) \equiv \text{const}.
$$
We now show that this implies $\mathbf{p} \mapsto \mu_j(\mathbf{p}) \equiv \text{const}$ for all $j = 1, \ldots, t$, contradicting the assumption that $\mu_j(\mathbf{p})$ is non-constant for some $j$. Formally, we prove the following equivalence:
\begin{equation}
\label{eq:induct}
\mathbf{p} \mapsto f_t(u,\mathbf{p}) \equiv \text{const} \quad \Longleftrightarrow \quad \mathbf{p}\mapsto\mu_i(\mathbf{p}) \equiv \text{const} \quad \text{for all } i = 1, \ldots, t.
\end{equation}
The $(\Longleftarrow)$ is trivial. For the other direction, we firstly recall the polynomial identities 
\begin{align*}
    &\prod_{i=1}^n (1+a_i x) =\sum_{r=0}^n e_r (a_1,\ldots,a_n)x^r,\\
&\prod_{i=1}^n (x - x_i) = x^n - e_1(x_1, \dots, x_n) x^{n-1} + e_2(x_1, \dots, x_n) x^{n-2} - \cdots + (-1)^n e_n(x_1, \dots, x_n),
\end{align*}
where \( e_r(a_1, \dots, a_n) \) denotes the \( r \)-the elementary symmetric polynomial, given by
\[
e_r(a_1, \dots, a_n) = \sum_{1 \le i_1 < \cdots < i_r \le n} a_{i_1} \cdots a_{i_r}\;.
\]
 for $r\geq 1$ and with $e_0(a_1, \ldots, a_n)=1$ 

The second identity is known as Vieta's identity, whereas for the first we refer to the reader to equation (2.2) in \cite{macdonald2008symmetric}).

Let $x := u^2$, and consider the function
\[
F_t(x, \mathbf{p}) := f_t(u,\mathbf{p}) = \prod_{j=1}^t \left(1 + 4x \mu_j^2(\mathbf{p})\right).
\]
This defines a polynomial of degree $t$ in $x$ with real coefficients, which can be expanded as
\[
F_t(x, \mathbf{p}) = \sum_{r=0}^t C_r(\mathbf{p}) x^r,
\]
where each coefficient $C_r(\mathbf{p})$ is given explicitly by
\[
C_r(\mathbf{p}) = 4^r \cdot e_r\left(\mu_1^2(\mathbf{p}), \dots, \mu_t^2(\mathbf{p})\right),
\]
with
$$e_r(\mu_1^2(\mathbf{p}), \ldots, \mu_t^2(\mathbf{p}))=\sum_{1\leq i_1< i_2< \cdots < i_r\leq t} \mu_{i_1}^2(\mathbf{p})\ldots\mu_{i_r}^2(\mathbf{p})\;.$$
Since $f_t(u,\mathbf{p})$ is independent of $\mathbf{p}$, each coefficient $C_r(\mathbf{p})$ is constant--this follows from the fact that two polynomials that agree for all $x$ must have identical coefficients. Hence, each $e_r(\mu_1^2(\mathbf{p}), \dots, \mu_t^2(\mathbf{p}))$ is also constant. Next, define the monic (the coefficient of the highest-degree term is equal to $1$) polynomial in $x$ 
\[
P(x, \textbf{p}) := \prod_{j=1}^t \left(x - \mu_j^2(\mathbf{p})\right) = x^t - e_1(\mu_1^2(\mathbf{p}), \dots, \mu_t^2(\mathbf{p})) x^{t-1} + \cdots + (-1)^t e_t(\mu_1^2(\mathbf{p}), \dots, \mu_t^2(\mathbf{p})).
\]
Since the coefficients $e_r(\mu_1^2(\mathbf{p}), \dots, \mu_t^2(\mathbf{p}))$ are constant in $\mathbf{p}$, the polynomial $P(x,\mathbf{p})$ is independent of $\mathbf{p}$. Because a polynomial is uniquely determined by its set of roots, it follows that the roots $\{\mu_1^2(\mathbf{p}), \dots, \mu_t^2(\mathbf{p})\}$ are constant in $\mathbf{p}$. 
In particular, for each $j$, there exists a constant $c_j \geq 0$ 
such that $\mu_j^2(\mathbf{p}) = \lambda_j$ for all $\mathbf{p} \in \Omega$. If $\lambda_j>0$, then $\mu_j(\mathbf{p})\in\{\pm\sqrt{c_j}\}$ for all $\mathbf{p}$. 
For any vector $\mathbf{x} \in \mathbb{R}^t$, it holds that 
\[
\mathbf{x}^\top B(\mathbf{p})\mathbf{x}=\mathbf{x}^\top \Sigma^{1/2}(\mathbf{p}) \Lambda(\mathbf{p}) \Sigma^{1/2}(\mathbf{p}) \mathbf{x}
= (\Sigma^{1/2}(\mathbf{p})  \mathbf{x})^\top \Lambda(\mathbf{p}) (\Sigma^{1/2}(\mathbf{p})  \mathbf{x}).
\]
Let $\mathbf{y} := \Sigma^{1/2}(\mathbf{p})\mathbf{x}$. Then
\[
\mathbf{y}^\top \Lambda(\mathbf{p}) \mathbf{y} = \sum_{i=1}^t \lambda_i(\mathbf{p}) y_i^2\leq 0
\]
since $ \lambda_i(\mathbf{p})\leq 0$; see \eqref{eq:lambda}.
Thus, $B(\mathbf{p})$ is negative  semi-definite. Therefore $\mathbf{p}\mapsto \mu_j(\mathbf{p})=-c_j$ for all $\mathbf{p}, $ i.e it must be constant for in $\mathbf{p}$ for all $j$. 

\end{proof}


\end{document}